\documentclass[12pt]{amsart}

\usepackage{MnSymbol}
\usepackage{mathtools}

\newcommand{\R}{\mathbb R}
\newcommand{\N}{\mathbb N}

\newcommand{\Z}{\mathbb Z}
\newcommand{\E}{\mathbb E}

\newtheorem{theorem}{Theorem}[section]
\newtheorem{lemma}[theorem]{Lemma}

\newtheorem{corollary}[theorem]{Corollary}
\newtheorem{proposition}[theorem]{Proposition}

\theoremstyle{definition}

\theoremstyle{remark}
\newtheorem{remark}[theorem]{Remark}

\newcommand{\rd}{{\mathbb R^d}}
\newcommand{\rr}{{\mathbb R}}

\newcommand{\rev}[1]{{\color{black} #1}}
\newcommand{\revnew}[1]{{\color{black} #1}}

\usepackage{comment}
\usepackage{tikz}
\usepackage{float}
\usepackage{relsize}
\usepackage{bbm}
\usepackage[T1]{fontenc}
\usepackage[english]{babel}
\usepackage{color}
\usepackage{hyperref}
\usepackage{graphicx}

\usepackage[capitalize]{cleveref}

\usepackage{xparse}
\let\realItem\item 
\makeatletter
\NewDocumentCommand\myItem{ o }{%
   \IfNoValueTF{#1}%
      {\realItem}
      {\realItem[#1]\def\@currentlabel{#1}}
}
\makeatother

\usepackage{enumitem}    
\setlist[enumerate]{
    before=\let\item\myItem,       
    label=\textnormal{(\arabic*)}, 
    widest=(2')                    
}

\usepackage[capitalize]{cleveref}

\date{\today}

\begin{document}

\sloppy

\title[Existence and regularity of local time]{On the existence and regularity of local times} 

\author{Tommi Sottinen}
\address{Tommi Sottinen, School of Technology and Innovations, University of Vaasa, PO Box 700, FI-65101 Vaasa, Finland}
\email{tommi.sottinen\@@{}uwasa.fi}

\author{Ercan S\"onmez}
\thanks{ES acknowledges financial support from the University of Klagenfurt}
\address{Ercan S\"onmez, Faculty of Mathematics, Ruhr University Bochum, 44780 Bochum, Germany}
\email{ercan.soenmez\@@{}rub.de}

\author{Lauri Viitasaari}
\address{Lauri Viitasaari, Aalto University School of Business, Department of Information and Service Management, PO Box 21210, 00076 Aalto, Finland}
\email{lauri.viitasaari\@@{}aalto.fi}

\begin{abstract}
We study the existence and regularity of local times for general $d$-dimensional stochastic processes. We give a general condition for their existence and regularity properties. To emphasize the contribution of our results, we show that they include various prominent examples, among others solutions to stochastic differential equations driven by fractional Brownian motion, where the behavior of the local time was not fully understood up to now and remained as an open problem in the stochastic analysis literature. In particular this completes the picture regarding the local time behavior of such equations, above all includes all possible dimensions and Hurst parameters. As other main examples, we also show that by using our general approach, one can quite easily cover and extend some recently obtained results on the local times of the Rosenblatt process and Gaussian quasi-helices.
\end{abstract}

\maketitle

\allowdisplaybreaks

\section{Introduction}
The local time of a stochastic process is the Radon-Nikodym derivative of the occupation measure with respect to the Lebesgue measure. The occupation measure measures the amount of time the path of the process spends in a given set. The local times of $d$-dimensional paths have generated much interest lately. This is motivated by the fact that in order to study the sample path irregularity properties of stochastic processes, one important aspect is the smoothness of their local times. Indeed, more regularity in the local time leads to less regularity in the sample paths, and vice versa. In addition, regularity of local times can be used to obtain existence of pathwise stochastic integrals with discontinuous integrands and irregular drivers, see \cite{HTV1,HTV2}.

The study of the existence of local times dates back to the works of Berman \cite{berman,berman2}. In his approach the existence of local times follows from the integrability properties of the Fourier transform.  More precisely, to prove the existence of local times  for $d$-dimensional stochastic processes $X=(X_t)_{t\in[0,T]}$ with $T \in (0, \infty)$ one needs to verify that
$$\int_\rd \int_{[0,T]} \int_{[0,T]} \Big\lvert \mathbb{E} \Big[ \exp \big( i \langle \xi, X_t-X_s\rangle \big) \Big] \Big\rvert ds dt d\xi <\infty,$$
where $\langle \cdot, \cdot \rangle$ denotes the scalar product on $\rd$. While the Fourier approach is very suitable also for studying the regularity of the local time, see e.g. \cite{xiao-ptrf,knsv}, it requires knowledge on the characteristic function. Consequently, this makes the study of local times of processes a difficult problem beyond the Gaussian case, in which the characteristic functions are not known. Recent contributions \cite{bouf2,bouf1} provide an approach under which certain analytical estimates on the absolute value of characteristic functions of the increments are shown to be efficient in order to conclude existence and regularity properties of local times. More precisely, in \cite{bouf2} the authors prove that local times admit certain Besov regularity, provided that the absolute value of the characteristic function (of arbitrary linear combinations of the process) satisfies a certain upper bound. In this article we borrow this idea to provide regularity results for the local times which we believe to be sharp. To emphasize the contribution of our results, we show that they include various prominent examples. These include, among others, an example where the behavior of the local time was not fully understood up to now and remained as an open problem in the stochastic analysis literature. Indeed, as an illustrative example, we apply our results to find regularity of local times for the solutions of certain stochastic differential equations driven by fractional Brownian motions. In particular this completes the picture regarding the behavior of the local time of such equations and complements all the gaps in \cite{lou}. We would like to mention that the approach in \cite{lou} relies in proving certain density estimates for the law of the solution to such equations. However, these estimates are shown not to be good enough in high dimensions and for large Hurst parameters. In this paper, we completely avoid the use of densities and overcome the aforementioned problems. Consequently, we achieve to obtain desired results both in higher dimensions and for large Hurst parameters.

For further illustration purposes, we also provide a simple proof for the fact that the Rosenblatt process satisfies our general conditions. This allows to recover and (slightly) improve recent results of \cite{knsv}, where the regularity of the local time of the Rosenblatt process was studied in detail, by means of harmonic analysis and operator theory. Finally, as a simple corollary we obtain regularity of the local times for Gaussian quasi-helices satisfying the Gaussian local non-determinism condition, that in particular covers some known results for the local times of $d$-dimensional fractional Brownian motions, see \cite{xiao-ptrf}.

Our main result, Theorem \ref{main}, is a development of some results in \cite{bouf2,bouf1}. We apply a similar upper bound for the absolute value of the characteristic function. In comparison, our required upper bound is more relaxed, and hence our general conditions are easier to verify. Most notably and unlike in \cite{bouf2}, our approach does not require densities of the process to be infinitely differentiable, cf. Remark \ref{remark:comparison}. Additionally, we provide further new regularity properties for the local time.

The rest of the paper is organized as follows.  In Section \ref{sect:rlt} we state our main results, and we provide the results concerning our examples in Section \ref{examples}. All the proofs are given in Section \ref{sect:proofs}.

\section{Regularity of local times}
\label{sect:rlt}
Throughout this paper let $T\in(0, \infty)$, $(\Omega, \mathbb{F}=(\mathcal{F}_t)_{t\in[0,T]},\mathbb{P})$ the underlying filtered probability space satisfying standard assumptions and suppose that $X=(X_t)_{t\in[0,T]}$ is an adapted process with values in $\rd$, $d\in\N$. Recall that the occupation measure of $X$ is defined for sets $A\times B \subset \rd \times [0,T]$ as
$$\tau_X (A,B) = \lambda \big( B \cap X^{-1} (A) \big),$$
where $\lambda$ denotes the Lebesgue measure on $[0,T]$, so that $\tau_X$ measures the time the process $X$ spends in the set $A$. If the measure $A\mapsto \tau_X(A,B)$ is absolutely continuous with respect to the $d$-dimensional Lebesgue measure $\lambda_d$, the corresponding Radon-Nikodym derivative
$$ L(x,B) = \frac{d\tau_X}{d\lambda_d} (x,B)$$
is called the local time (or occupation density) of $X$. In other words, we have
$$\tau_X (A,B) = \int_A L(x,B) dx.$$
For simplicity we denote $L(x,t) = L(x,[0,t])$ {and note that, for a closed interval $I=[a,b]$, we have $L(x,I) = L(x,b)-L(x,a)$.}

\subsection{General setting with main results} \label{s21}
The well-known approach due to Berman \cite{berman} in proving the existence of local times relies on integrability properties of the Fourier transform. Generally one is left to show that
$$\int_\rd \int_{[0,T]} \int_{[0,T]} \Big\lvert  \mathbb{E} \Big[ \exp \big( i \langle \xi, X_t-X_s\rangle \big) \Big] \Big\rvert  ds dt d\xi <\infty,$$
where $\langle \cdot, \cdot \rangle$ denotes the scalar product on $\rd$. Beyond the Gaussian case, in which the characteristic functions are not known, this makes the study of local times of processes as considered here a challenging problem. Recent contributions \cite{bouf2,bouf1} provide an approach under which certain analytical estimates on the absolute value of characteristic functions of the increments are shown to be efficient in order to conclude existence and regularity properties of local times. This is the idea we borrow in the following in order to conclude regularity results which we believe to be sharp. As our main example, we apply these results to processes for which the regularity of local times were an open problem up to now. Throughout, we denote by $\|\cdot\|$ the Euclidean norm on $\rd$. Our main theorem is the following.
\begin{theorem} \label{main}
Suppose that $X=(X_t)_{t\in[0,T]}$ is a stochastic process with values in $\rd$ and continuous paths \textcolor{black}{such that $X_0$ is deterministic}. Assume that the following conditions are satisfied for some $\alpha \in (0,\frac{1}{d})$.
\begin{enumerate}
\item[(A1)]
\rev{For every $m\in \N$, $\xi_j = (\xi_{j,1}, \ldots, \xi_{j,d}) \in (\rr\setminus \{0\} )^d$, $k_{j,l} \in \{0,4\}$, $1\leq j\leq m$ and $1\leq l\leq d$, and $0=t_0  <t_1 < \ldots < t_m<T$,} it holds
\begin{align*}
\Big\lvert  \mathbb{E} \Big[ \exp \Big( i \sum_{j=1}^m \langle \xi_j, X_{t_{j}}-X_{t_{j-1}}\rangle \Big) \Big]\Big\rvert  \leq C^mm^{\alpha \theta m}\prod_{j=1}^m \prod_{l=1}^d\frac{1}{\lvert\xi_{j,l}\rvert^{k_{j,l}} 
(t_j-t_{j-1})^{\alpha k_{j,l}}},
\end{align*}
where $C>0$ and $\theta \geq 0$. \label{a1}
\item[(A2)] 
There exists a constant $C>0$ and $\iota \in[0,1]$ such that for any $p\geq 1$ and any $0 \leq s \leq t \leq T$ it holds
$$\mathbb{E} [ \| X_t - X_s\| ^p] \leq C^pp^{p\iota} \lvert t-s\rvert^{p\alpha}.$$ \label{a2}
\end{enumerate}
Then the local time $(x,t) \mapsto L(x,[0,t])$ exists \textcolor{black}{and satisfies, for every \textcolor{black}{$n\in \mathbb{N}$}, $0\leq s< t\leq T$, $x,y \in \mathbb{R}^d$, and every $0\leq \gamma<\min\left(1,\frac{1-\alpha}{2\alpha}\right)$, that
\begin{equation}\label{eq:mom3-theorem}
\E \lvert L(x,t)-L(x,s)\rvert^n \leq C^n (t - s)^{(1 - \alpha d) n} n^{ n \alpha (d+\theta) }
\end{equation}
and 
\begin{equation}\label{eq:mom4-theorem}
\left\lvert \E (L(x+y,[s,t])-L(x,[s,t]))^{n} \right\rvert  \leq C^n \|y\|^{\gamma n} (t - s)^{(1 - \alpha(d + \gamma) ) n} n^{ n \alpha( d + \gamma+\theta) },
\end{equation}
where the constant $C$ depends solely on $\gamma$, $d$, and $\alpha$. 
In particular, $L$ is jointly continuous almost surely.} \textcolor{black}{Let $I \subset (0,T)$ be a closed interval and set $L^*(I) = \sup_{x\in \rd} L(x,I)$. Then} there exist deterministic constants $C_1$ and $C_2$ such that, almost surely,
\begin{equation}
\label{eq:main-fixed-s}
 \limsup _{r \to 0} \frac{\textcolor{black}{ L^*([s-r,s+r])}}{r^{1-\alpha d} (\log \log r^{-1})^{\alpha(d+\theta)}} \leq C_1
 \end{equation}
for \textcolor{black}{every $s\in I$} and
\begin{equation}
 \label{eq:main-sup-s}
\limsup _{r \to 0} \sup_{s\in \textcolor{black}{I}}  \frac{ \textcolor{black}{L^*([s-r,s+r])}}{r^{1-\alpha d} (\log r^{-1})^{\alpha(d+\theta)}} \leq C_2 .
\end{equation}
\end{theorem}

\textcolor{black}{
\begin{remark}
\label{remark:away-from-zero}
In the above result one can replace the condition $0=t_0  <t_1 < \ldots < t_m<T$ in $(A1)$ with the condition 
$0< \epsilon \leq t_1 < \ldots < t_m<T$ (with the convention $t_0=0$) for some $\epsilon$ (in this case the constant $C$ may depend on $\epsilon$ as well). In this case the local time $L(x,[\epsilon,t])$ for $t\in [\epsilon,T]$ exists and is jointly continuous and  \eqref{eq:main-fixed-s} and \eqref{eq:main-sup-s} hold for $s\in I \subset (\epsilon,T)$. Note also that if this replaced condition holds for every $\epsilon>0$, then by a limiting argument one obtains the existence of the local time $L(x,[0,t])$, cf. Theorem \ref{main2} below. 
\end{remark}
}
\begin{remark}
The parameter $\theta\geq 0$ in (A1) is arbitrary, and hence allows very rapid growth in the constant (in terms of $m$). We point out that in typical cases, one can choose $\theta=0$ which will be the case in almost all of our examples. However, we allow more general growth in order to emphasize the generality of our results. The parameter $\iota$ in Assumption (A2) on the other hand is related to the sharper modulus of continuity involving a logarithmic term, and the exponent for the logarithmic term arises from $\iota$, which on the other hand is related to the existence of certain exponential moments, see Proposition \ref{prop:sup-tail} below. In the Gaussian case, it is well-known that one can choose $\iota = 1/2$, see Proposition \ref{main4}, while in the Rosenblatt case one chooses $\iota=1$, see Proposition \ref{main3}.
\end{remark}
\begin{remark}
\label{remark:comparison}
The authors in \cite{bouf2,bouf1} assume a condition similar to (A1), called $\alpha$-local nondeterminism, see \cite[Definition 2.4]{bouf1} and \cite[Definition 2.8]{bouf2}. The concept of this notion is explained as an extension of local nondeterminism in the framework of Gaussian and stable processes \cite{berman2,nolan}, connected to the (un)predictability of paths. The main difference is that in \cite{bouf2} the authors assumed that one can choose integers $k_{j,l}$ arbitrarily, while here we merely assume two possible choices $k_{j,l} \in \{0,4\}$. It is worth to point out that our condition is not as restrictive, as it allows to cover densities that are not infinitely differentiable. Indeed, by using the well-known relation between smoothness of the density and the decay of the Fourier transform, choice $k_{j,l}=4$ requires only existence of derivatives up to fourth order. In this article, we show that condition (A1) as stated here along with (A2) is enough to obtain existence of the local time together with additional (sharp) regularity properties. In order to obtain sharp regularity estimates, it is also required to keep track on the constant depending on $m$ (in our case, in terms of $\theta$). This was omitted in \cite{bouf2} where the authors did not study sharp regularity. Finally, we also point out that our condition is global in the time points $0=t_0<t_1<\ldots<t_m<T$, while the condition in \cite{bouf2} is stated on small time scales $t_m-t_0<\epsilon$ only. However, by examining our proof carefully one observes that our results remain valid under local conditions stated for small time scales only. For the sake of simplicity, we assume the global condition (A1) in order to avoid unnecessary extra technicalities in our proofs.
\end{remark}

Following \cite{knsv} we also get the following corollaries.

\begin{corollary} \label{cm1}
Under the assumptions of Theorem \ref{main} there exist deterministic constants $C_1$ and $C_2$, not depending on $x$ and $t$, such that for every \textcolor{black}{$t\in I$} and $x \in \rd$ we have, almost surely,
$$ \limsup _{r \to 0} \frac{ L(x,[t-r,t+r])}{r^{1-\alpha d} (\log \log r^{-1})^{\alpha(d+\theta)}} \leq C_1$$
and for every $x \in \rd$ we have, almost surely,
$$ \limsup _{r \to 0} \sup_{t \in \textcolor{black}{I}} \frac{ L(x,[t-r,t+r])}{r^{1-\alpha d} (\log r^{-1})^{\alpha(d+\theta)}} \leq C_2.$$
\end{corollary}

\begin{corollary} \label{cm2}
Under the assumptions of Theorem \ref{main} there exists a deterministic constant $C \in (0, \infty)$ such that for every $s\in \textcolor{black}{I}$ we have, almost surely, 
$$ \liminf _{r \to 0} \sup_{t \in (s-r,s+r)} \frac{\|X_t - X_s \| }{r^{\alpha d} (\log \log r^{-1})^{-\alpha(d+\theta)}} \geq C$$
and 
$$ \liminf _{r \to 0} \inf_{s\in \textcolor{black}{I}} \sup_{t \in (s-r,s+r)} \frac{\|X_t - X_s \| }{r^{\alpha d} ( \log r^{-1})^{-\alpha(d+\theta)}} \geq C.$$
In particular, $X$ is almost surely nowhere differentiable.
\end{corollary}

Before turning to the most interesting examples as a highlight of this paper, we would like to remark that due to the contribution \cite{bouf2} one can also conclude certain Besov regularities (see \cite[Theorem 1.1, Theorem 1.4 and Corollary 1.5]{bouf2} for details) and the fact that $L(x,t)$ has a continuous version, say $\tilde{L} (x,t)$ satisfying
$$ \lvert\tilde{L} (x,t) - \tilde{L} (x,s) \rvert \leq \eta \lvert t-s\rvert^\beta$$
for every $\beta \in (0, 1-d\alpha)$ and some random variable $\eta$.

\subsection{Examples}
\label{examples}
We begin with the leading example, which covers the case where $X$ is the solution of stochastic differential equations driven by $d$-dimensional fractional Brownian motion. In particular, we show that one can apply Theorem \ref{main}, allowing us to complement the open gaps in \cite{lou}. \textcolor{black}{In order to establish (A1) and as in \cite{lou}, we apply estimates for the Malliavin derivatives proved in \cite{baudoin}. However, these estimates force us to restrict ourselves away from zero, cf. Theorem \ref{main2} below and Remark \ref{remark:away-from-zero}.}

Let $B=(B^1, \ldots, B^d)$ be a $d$-dimensional fractional Brownian motion with Hurst parameter $H \in (\frac{1}{4}, 1)$, i.e. the components $B^l = (B^l_t)_{t\in [0,T]}$, $1 \leq l \leq d$, are independent centered Gaussian processes satisfying
$$\mathbb{E} [ (B_t^l - B_s^l) ^2 ] = \lvert t-s\rvert^{2H}, \quad \forall s,t \in [0,T].$$
Consider the following class of differential equations given by
\begin{equation}\label{SDE1}
X_t = x + \int_0^t V_0 (X_s) ds + \sum_{l=1}^d \int_0^t V_l(X_s) dB_s^l,
\end{equation}
where $t \in [0,T]$, $x \in \rd$ is the initial condition and $V_0, V_1, \ldots, V_d$ are smooth vector fields in $\rd$. Of importance is to mention that the stochastic integrals are understood in the Young sense for $H>\frac12$ (see \cite{zahle}) and in the rough path sense for $\frac{1}{4} < H < \frac12$ (see \cite{friz}). Regarding the well-posedness of these stochastic differential equations we state the following assumptions which are also standard in studying such equations (see \cite{baudoin,lou}).

Assume that $V_0, V_1, \ldots, V_d \in \mathcal{C}_b^\infty (\rd)$, i.e. they possess bounded derivatives of all orders, and suppose that $V_0, V_1, \ldots, V_d$ satisfy the uniform elliptic condition
$$ v^{\operatorname{T}} V(x) V(x)^{\operatorname{T}} v \geq \lambda \| v\| ^2 \quad \forall x,v \in \rd,$$
where $V = (V_j^i)_{i,j = 1 \ldots, d}$ and $\lambda \in (0, \infty)$ is some constant.

Having specified the class of stochastic differential equations we now state the following result, which considerably extends the regularity results in \cite[Theorem 1.1]{lou} as it includes all possible dimensions and Hurst parameters for which the local time exists.

\begin{theorem} \label{main2}
Let $X=(X_t)_{t\in[0,T]}$ be the solution to \eqref{SDE1} with the assumptions above \textcolor{black}{and suppose $dH<1$. Then the local time $L(x,[0,t])$ exists. Let $I\subset (0,T)$ be a closed interval. Then we also have \eqref{eq:mom3-theorem} and \eqref{eq:mom4-theorem} with $\alpha =H$ and \textcolor{black}{$\theta= \frac{8d+\delta}{H}$}, for any $\delta>0$, provided that $s,t\in I$. In particular, the local time $L(x,[0,t]\cap I)$ is jointly continuous almost surely. Moreover, \eqref{eq:main-fixed-s}, \eqref{eq:main-sup-s}, and the assertions of Corollary \ref{cm1} and Corollary \ref{cm2} hold. }
\end{theorem}
\textcolor{black}{
\begin{remark}
Note that in the above result, we always have sharp H\"older estimates \eqref{eq:main-fixed-s} and \eqref{eq:main-sup-s} similarly as in Theorem \ref{main}, as in these statements one already considers an interval $[s-r,s+r]$ that is bounded away from zero. In comparison to Theorem \ref{main}, the only difference is that one obtains joint continuity only on arbitrary subintervals $I$ that are automatically bounded away from zero. However, by a limiting argument we still obtain the existence of $L(x,[0,t])$ over the whole interval.
\end{remark}
}
\textcolor{black}{
\begin{remark}
\textcolor{black}{It is evident that the value of $\theta$ is not optimal.} We note however that this does not play any significant role as it affects only to the exponent in the logarithmic term in Corollary \ref{cm1} and \ref{cm2}.
\end{remark}
}
As a next example of our main results in Section \ref{s21} we would like to add the Rosenblatt process $Z=(Z_t)_{t\in[0,T]}$ which depends on an underlying Hurst index $H\in (\frac12, 1)$ and is treated in \cite{knsv}. In particular, for this process one has the representation
$$Z_t = \int_{\{ (x,y) \in \rr^2 : x \neq \pm y \} } \frac{e^{it (x+y)} -1}{i(x+y)} Z_G(dx) Z_G(dy)$$
for every $t \in [0,T]$, where $Z_G(dx)$ is a complex-valued random white noise with control measure $G$ satisfying $G(dx) = \lvert x\rvert^{-H} dx$. See \cite{knsv} for a more detailed introduction and discussion of this process. As a special case of Theorem \ref{main} we obtain the following results, also addressed in \cite{knsv}.

\begin{proposition} \label{main3}
Let $Z=(Z_t)_{t\in[0,T]}$ be the Rosenblatt process on $\mathbb{R}$ with $H\in (\frac12, 1)$. Then the assertions of Theorem \ref{main}, Corollary \ref{cm1} and Corollary \ref{cm2} are true with $\alpha =H$, $\theta=0$ and $\iota=1$.
\end{proposition}
\begin{remark}
The above proposition improves the results of \cite{knsv} slightly in terms of the exponent in the logarithmic term. That is, we observe
\begin{equation*}
 \limsup _{r \to 0} \frac{\sup_{x\in\rr} L(x,[s-r,s+r])}{r^{1-H} (\log \log r^{-1})^{H}} \leq C_1,
 \end{equation*}
while in \cite{knsv} the authors observe
\begin{equation*}
 \limsup _{r \to 0} \frac{\sup_{x\in\rr} L(x,[s-r,s+r])}{r^{1-H} (\log \log r^{-1})^{2H}} \leq C_1.
 \end{equation*}
\end{remark}
Finally we show that Theorem \Ref{main} also covers the popular class of certain Gaussian processes in which case the results are already well-known, see e.g. \cite[Corollary 1.1]{xiao-ptrf}.

\begin{proposition} \label{main4}
Let $Z=(Z_t)_{t\in[0,T]}$ be a centered $d$-dimensional Gaussian process with components $Z^l$ satisfying
\begin{equation}
\label{eq:Gaussian-variance}
C_- \lvert t-s\rvert^{2H} \leq \mathbb{E}(Z^l_t - Z^l_s)^2 \leq C_+\lvert t-s\rvert^{2H}
\end{equation}
for some $H \in (0,\frac{1}{d})$. Suppose further that the local non-determinism property
\begin{equation}
\label{eq:Gaussian-LND}
\E \left\Vert \sum_{k=1}^m \langle \xi_k,Z_{t_k}-Z_{t_{k-1}}\rangle\right\Vert^2 \geq C \prod_{l=1}^d \sum_{k=1}^m \xi_{k,l}^2 \E (Z^l_{t_k}-Z^l_{t_{k-1}})^2
\end{equation}
holds for all $m\in \mathbb{N}$, $\xi_k \in \mathbb{R}^d$ and all $0=t_0<t_1<\ldots<t_m\leq T$. Then the assertions of Theorem \ref{main}, Corollary \ref{cm1} and Corollary \ref{cm2} are true with $\alpha=H$, $\theta=0$, and $\iota=\frac12$.
\end{proposition}
\begin{remark}
The above result covers, e.g. the case of $d$-dimensional fractional Brownian motion with Hurst index $H<\frac{1}{d}$, proved already in \cite{xiao-ptrf}. The case of the fractional Brownian motion is obviously also already covered by Theorem \ref{main2}.
\end{remark}
\section{Proofs}
\label{sect:proofs}
\subsection{Proof of the main results}
We begin by proving the existence of local time.
\begin{proposition}
The local time $L$ of $X$ exists and has a representation
\begin{align}
\label{eq:local_rep} L(x, t) = \frac{1}{2 \pi} \int_{\rd} \int_0^t e^{i \langle\xi, (x - X_s)\rangle } ds d\xi.
\end{align}
\end{proposition}

\begin{proof}
The existence and representation \eqref{eq:local_rep} follows from \cite{berman} if we show 
\begin{align*}
&\int_\rd \int_{[0,T]} \int_{[0,T]} \Big\lvert  \mathbb{E} \Big[ \exp \big( i \langle \xi, X_t-X_s\rangle \big) \Big] \Big\rvert  ds dt d\xi \\
&= 2 \int_0^T \int_0^t \int_\rd \Big\lvert  \mathbb{E} \Big[ \exp \big( i \langle \xi, X_t-X_s\rangle \big) \Big] \Big\rvert  d\xi ds dt  <\infty.
\end{align*}
We set $I_1 = \left[-(t-s)^{-\alpha},(t-s)^{-\alpha}\right]$ and $I_2 = \mathbb{R} \setminus I_1$. Then
$$
\mathbb{R}^d = \bigcup_{i_{l}\in \{1,2\}, l=1,\ldots,d} \prod_{l=1}^d I_{i_{l}}.
$$
Next we apply (A1) to get 
$$
\Big\lvert  \mathbb{E} \Big[ \exp \big( i \langle \xi, X_t-X_s\rangle \big) \Big] \Big\rvert  \leq C\prod_{l=1}^d\lvert\xi_l\rvert^{-k_{l}}(t-s)^{-\alpha k_{l}},
$$
where we choose $k_{l} = k_{l}(i_{l})=0$ if $i_{l}=1$ and $k_{l}(i_{l})=4$ if $i_{l}=2$. Then 
\begin{align*}
&\int_0^T \int_0^t \int_\rd \Big\lvert  \mathbb{E} \Big[ \exp \big( i \langle \xi, X_t-X_s\rangle \big) \Big] \Big\rvert  d\xi ds dt \\
& \leq C\sum_{i_{l}\in \{1,2\}, l=1,\ldots,d} \int_0^T \int_0^t \int_{\prod_{l=1}^d I_{i_{l}}}\prod_{l=1}^d\lvert\xi_l\rvert^{-k_{l}}(t-s)^{-\alpha k_{l}} d\xi dsdt,
\end{align*}
where 
$$
\int_{I_{i_l}} \lvert\xi_l\rvert^{-k_{l}}(t-s)^{-\alpha k_{l}} d\xi_l = C(t-s)^{-\alpha}
$$
by our choices of $k_l$ for $i_l$. Hence 
\begin{align*}
&\int_0^T \int_0^t \int_\rd \Big\lvert  \mathbb{E} \Big[ \exp \big( i \langle \xi, X_t-X_s\rangle \big) \Big] \Big\rvert  d\xi ds dt \\
& \leq C\sum_{i_{l}\in \{0,2\}, l=1,\ldots,d} \int_0^T \int_0^t (t-s)^{-d\alpha} dsdt < \infty
\end{align*}
which concludes the proof.
\end{proof}
\begin{proposition}\label{prop:int_fourier} Let $n \in \mathbb{N}$ and  $0 \leq \eta < \min\left(1,\frac{1-\alpha d}{2\alpha}\right)$. Then, for any times $0 < u < U\leq T$ we have
\begin{align}
\label{eq:bound_fourier}   \int_{[u, U]^n} \int_{(\mathbb{R}^{d})^n} \prod_{j =1}^n \|\xi_j\|^{\eta} \left\lvert \mathbb{E} \exp \left( i \sum_{j = 1}^n \langle \xi_j, X_{t_j}\rangle \right) \right\rvert d\xi dt  \leq C^n n^{n\alpha(d + \eta+\theta)} (U - u)^{(1 - \alpha(d + \eta) ) n},
\end{align} 
\noindent where the constant $C > 0$ depends only on $\alpha, d$, and $\eta$. 
\end{proposition} 
\begin{proof}
\rev{The proof is similar to \cite[Lemma 3,2]{bouf2}. In the sequel, we denote by $C$ a generic unimportant constant that may vary from line to line.} Performing \textcolor{black}{the} change of variables $\xi'_j \coloneqq \sum_{\ell=j}^n\textcolor{black}{\xi_\ell}$, $j=1,\ldots ,n$ with the convention \textcolor{black}{$\xi'_{n+1}=0$}, we see that
\begin{align*}
I &:= \int_{[u, U]^n} \int_{(\mathbb{R}^{d})^n} \prod_{j =1}^n \|\xi_j\|^{\eta} \left\lvert \mathbb{E} \exp \left( i \sum_{j = 1}^n \langle\xi_j, X_{t_j}\rangle \right) \right\rvert d\xi dt  \\
&= n!   \int_{u\leq t_1<t_2<\ldots < t_n\leq U} \int_{(\mathbb{R}^{d})^n} \prod_{j =1}^n \|\xi_j\|^{\eta} \left\lvert \mathbb{E} \exp \left( i \sum_{j = 1}^n \langle\xi_j, X_{t_j}\rangle \right) \right\rvert d\xi dt  \\
&=C(d)n!   \int_{u\leq t_1<t_2<\ldots < t_n\leq U}\\
& \quad \int_{(\mathbb{R}^{d})^n} \prod_{j =1}^n \|\xi'_j-\xi'_{j+1}\|^{\eta} \left\lvert \mathbb{E} \exp\left(i\xi'_1 X_0\right)\exp \left( i \sum_{j = 1}^n \langle\xi'_j, (X_{t_j}-X_{t_{j-1}})\rangle \right) \right\rvert d\xi' dt\\
&\textcolor{black}{\leq} C(d)n!   \int_{u\leq t_1<t_2<\ldots < t_n\leq U} \int_{(\mathbb{R}^{d})^n} \prod_{j =1}^n \|\xi'_j-\xi'_{j+1}\|^{\eta} \left\lvert \mathbb{E} \exp \left( i \sum_{j = 1}^n \langle\xi'_j, (X_{t_j}-X_{t_{j-1}})\rangle \right) \right\rvert d\xi' dt.
\end{align*} 
\textcolor{black}{Here we have used the fact that the determinant of the Jacobian related to the change of variables $\xi'_j \coloneqq \sum_{\ell=j}^n\textcolor{black}{\xi_\ell}$ is a constant $C(d)$ depending only on $d$ (note that in the case $d=1$, we actually have $C(d) = 1$), and that 
$$
\sum_{j=1}^n \langle\xi_j, X_{t_j} \rangle= \langle\xi'_1 ,X_0\rangle + \sum_{j=1}^n \langle\xi_j , X_{t_j}-X_0\rangle = \langle\xi'_1, X_0\rangle + \sum_{j=1}^n \langle\xi'_j ,X_{t_j}-X_{t_{j-1}}\rangle.
$$
}
Here, \textcolor{black}{since $\eta\in(0,1)$},  $\|\xi'_j-\xi'_{j+1}\|^{\eta}\textcolor{black}{\leq} (\|\xi'_j\|^\eta+\|\xi'_{j+1}\|^\eta)$ and thus
$$
\prod_{j=1}^n \|\xi'_j-\xi'_{j+1}\|^{\eta}\textcolor{black}{\leq} \sum \prod_{j=1}^n \|\xi'_j\|^{\gamma_j},
$$ 
where the exponents in each term of the sum satisfy $\gamma_j\in \{0,\eta,2\eta\}$ and $\sum_{j=1}^n \gamma_j = \eta n$. Since $\xi'_{n+1}=0$, the number of summands equals $2^{n-1}$. This gives
\begin{align*}
I &\leq n!C^n   \int_{u\leq t_1<t_2<\ldots < t_n\leq U} \int_{\mathbb{R}^n} \sum\prod_{j =1}^n \|\xi'_j\|^{\gamma_j} \left\lvert \mathbb{E} \exp \left( i \sum_{j = 1}^n \langle\xi'_j, (X_{t_j}-X_{t_{j-1}})\rangle \right) \right\rvert d\xi' dt.
\end{align*}
By (A1), we have 
$$
\left\lvert \mathbb{E} \exp \left( i \sum_{j = 1}^n \langle\xi'_j, (X_{t_j}-X_{t_{j-1}})\rangle \right) \right\rvert \leq C^n n^{\alpha \theta n}\prod_{m=1}^n\prod_{l=1}^d \lvert\xi'_{m,l}\rvert^{-k_{m,l}}(t_m-t_{m-1})^{-\alpha k_{m,l}},
$$
where each $k_{m,l} \in \{0,4\}$. 
Again, define $I_1^m = \left[-(t_m-t_{m-1})^{-\alpha},(t_m-t_{m-1})^{-\alpha}\right]$ and $I_2^m = \mathbb{R} \setminus I_1^m$. Then
$$
(\mathbb{R}^d)^n = \bigcup_{i_{m,l}\in \{1,2\}, m=1,\ldots,n, l=1,\ldots,d} \prod_{m=1}^n\prod_{l=1}^d I^m_{i_{m,l}}.
$$
Now we choose $k_{m,l} = k_{m,l}(i_{m,l})=0$ if $i_{m,l}=1$ and $k_{m,l}(i_{m,l})=4$ if $i_{m,l}=2$. Using also 
$$
\prod_{j =1}^n \| \xi'_j \|^{\gamma_j} \leq \sum_{l_1,\ldots,l_d\in\{1,\ldots,d\}}\prod_{j=1}^n \lvert\xi'_{j,l_j}\rvert^{\gamma_j},
$$ 
it follows that 
\begin{align*}
I &\leq n!C^nn^{\alpha \theta n}  \sum_{l_1,\ldots,l_d\in\{1,\ldots,d\}}\sum_{i_{m,l}\in\{1,2\},m=1,\ldots,n,l=1,\ldots,d} \sum_{(\gamma_1,\ldots,\gamma_n)\in\{0,\eta,2\eta\}^n}  \\
& \int_{u\leq t_1<t_2<\ldots < t_n\leq U}\int_{\prod_{m=1}^n\prod_{l=1}^d I^m_{i_{m,l}}} \prod_{j =1}^n \lvert\xi'_{j,l_j}\rvert^{\gamma_j} \prod_{m=1}^n\prod_{l=1}^d \lvert\xi'_{m,l}\rvert ^{-k_{m,l}}(t_m-t_{m-1})^{-\alpha k_{m,l}} d\xi' dt.
\end{align*}
Following \textcolor{black}{\cite[pages 16-18]{bouf2}}, integrating in the $\xi'$ variables gives 
\begin{align*}
I &\leq n!C^nn^{\alpha \theta n}  \sum_{l_1,\ldots,l_d\in\{1,\ldots,d\}}\sum_{i_{m,l}\in\{1,2\},m=1,\ldots,n,l=1,\ldots,d} \sum_{(\gamma_1,\ldots,\gamma_n)\in\{0,\eta,2\eta\}^n} \\
& \int_{u\leq t_1<t_2<\ldots < t_n\leq U} \prod_{j=1}^n (t_j-t_{j-1})^{-\alpha (d+\gamma_j)} dt \\
&\leq n!C^n n^2 n^{\alpha \theta n}\sum_{(\gamma_1,\ldots,\gamma_n)\in\{0,\eta,2\eta\}^n}
\int_{u\leq t_1<t_2<\ldots < t_n\leq U} \prod_{j=1}^n (t_j-t_{j-1})^{-\alpha (d+\gamma_j)} dt,
\end{align*}
where the last inequality follows from the fact that 
$$
\sum_{l_1,\ldots,l_d\in\{1,\ldots,d\}} 1 = C(d)
$$
is just a constant depending on the fixed dimension $d$ only, while 
$$
\sum_{i_{m,l}\in\{1,2\},m=1,\ldots,n,l=1,\ldots,d} \leq C(d) \binom{n}{2} \leq C(d) n^2.
$$
Note that, for any $\theta_1,\theta_2>-1$, we have
\begin{equation}
\label{eq:beta-function}
\int_0^{t_{j+1}} (t_{j+1}-t_j)^{\theta_1}t_j^{\theta_2}dt_j = t_{j+1}^{1+\theta_1+\theta_2}B(1+\theta_1,1+\theta_2),
\end{equation}
where $B(x,y) = \int_0^1 t^{x-1}(1-t)^{y-1}dt$ denotes the Beta function. Hence, denoting $\theta_j = -\alpha(d+\gamma_j)$ and integrating in the order $t_1,t_2,\ldots, t_m$ and using 
\eqref{eq:beta-function} repeatedly, yields
\begin{align*}
&\int_{u\leq t_1<t_2<\ldots < t_n\leq U} \prod_{j=1}^n (t_j-t_{j-1})^{-\alpha (d+\gamma_j)} dt\\
&=\int_{0\leq t_1<t_2<\ldots < t_n\leq U-u}\prod_{j=1}^n (t_j-t_{j-1})^{\theta_j} dt\\
&=B(1+\theta_1,1+\theta_2)\int_{0\leq t_2<\ldots < t_n\leq U-u}\prod_{j=3}^n (t_j-t_{j-1})^{\theta_j}
t_2^{1+\theta_1+\theta_2} dt_2\ldots dt_n\\
&=B(1+\theta_1,1+\theta_2)B(2+\theta_1+\theta_2,1+\theta_3)\\
&\cdot\int_{0\leq t_3<\ldots < t_n\leq U-u}\prod_{j=4}^n (t_j-t_{j-1})^{\theta_j}t_3^{2+\theta_1+\theta_2+\theta_3} dt_3\ldots dt_n\\
&=\ldots\\
&=\prod_{j=2}^n B(j-1+\sum_{k=1}^{j-1}\theta_k,1+\theta_{j})\int_{0}^{U-u}t_n^{n-1+\sum_{k=1}^n \theta_j}dt_n \\
&=\prod_{j=2}^n B(j-1+\sum_{k=1}^{j-1}\theta_k,1+\theta_{j}) \frac{(U-u)^{n+\sum_{k=1}^n \theta_j}}{n+\sum_{k=1}^n \theta_j}
\end{align*}
since $\theta_j = -\alpha(d+\gamma_j) > -1$ as $\gamma_j\leq 2\eta < \frac{1-\alpha d}{\alpha} = \frac{1}{\alpha} -d$.
Recalling that
$$
B(x,y) = \frac{\Gamma(x)\Gamma(y)}{\Gamma(x+y)},
$$
where $\Gamma(x)$ is the Gamma function, we observe
$$
\prod_{j=2}^n B(j-1+\sum_{k=1}^{j-1}\theta_k,1+\theta_{j}) = \frac{\prod_{j=1}^n \Gamma(1+\theta_j)}{\Gamma(n+\sum_{k=1}^n \theta_j)}.
$$
Plugging in $\theta_j = -\alpha(d+\gamma_j)$, using the fact that $\sum_{j=1}^n \theta_j = -\alpha d n - \alpha \sum_{j=1}^n \gamma_j$, where $\sum_{j=1}^n \gamma_j = \eta n$, $n^2\leq C^n$, and collecting all the estimates gives 
\begin{align*}
I &\leq n!C^n n^2 n^{\alpha \theta n}\sum_{(\gamma_1,\ldots,\gamma_n)\in\{0,\eta,2\eta\}^n} \frac{\prod_{j=1}^n \Gamma(1+\theta_j)}{\Gamma((1-\alpha d -\alpha \eta)n)}\frac{(U-u)^{(1-\alpha d - \alpha \eta)n}}{(1-\alpha d -\alpha\eta)n}\\
&\leq n!C^nn^{\alpha \theta n} \frac{\prod_{j=1}^n \Gamma(1+\theta_j)}{\Gamma((1-\alpha d -\alpha \eta)n)}\frac{(U-u)^{(1-\alpha d - \alpha \eta)n}}{(1-\alpha d -\alpha\eta)n}\\
&\leq C^nn^{\alpha \theta n}  \frac{\Gamma(n+1)}{\Gamma((1-\alpha d -\alpha \eta)n)}(U-u)^{(1-\alpha d - \alpha \eta)n},
\end{align*}
where in the last inequality we have also used the definition of $\theta_j$ and $n! = \Gamma(n+1)$. 
Using Stirling's approximation gives 
$$
\frac{\Gamma(n+1)}{\Gamma((1-\alpha d -\alpha \eta)n)} \leq C^n n^{\alpha(d+\eta)n} ,
$$
which yields
$$
I \leq C^n n^{\alpha(d+\eta+\theta)n} (U-u)^{(1-\alpha d - \alpha \eta)n}
$$
and completes the proof.
\end{proof}

\textcolor{black}{The following proposition establishes the claimed bounds \eqref{eq:mom3-theorem} and \eqref{eq:mom4-theorem}.}
\begin{proposition}
\label{prop:moment-bounds}
For every $0 \leq s < t$ and $x \in \rd$, 
\begin{equation}\label{eq:mom3}
\E \lvert L(x,t)-L(x,s)\rvert^n \leq C^n (t - s)^{(1 - \alpha d) n} n^{ n \alpha (d+\theta) }.
\end{equation}
Moreover, for every $0\leq \gamma<\min\left(1,\frac{1-\alpha}{2\alpha}\right)$ and $y \in \rd$, we have
\begin{equation}\label{eq:mom4}
\left\lvert \E (L(x+y,[s,t])-L(x,[s,t]))^{n} \right\rvert  \leq C^n \|y\|^{\gamma n} (t - s)^{(1 - \alpha(d + \gamma) ) n} n^{ n \alpha( d + \gamma+\theta) }. 
\end{equation}
In both inequalities the constant $C$ depends only on $\gamma$, $d$, and $\alpha$.
\end{proposition}
\begin{proof}
First, by~\eqref{eq:local_rep}, 
\begin{align}
\notag & \E \big(L (x+y,[s,t])-L(x,[s,t])\big)^{n} \\
\notag  = & (2 \pi)^{-n} \int_{(\mathbb{R}^d)^n} \int_{[s, t]^n} \left( \prod_{j =1}^n ( \exp ( i \langle\xi_j, x+ y\rangle ) - \exp( i \langle\xi_j, x\rangle) ) \right)\E \exp\left( - i \sum_{j = 1}^n \langle\xi_j, X_{v_j}\rangle \right) dv d\xi.
\end{align}
Using $\gamma \in [0,1 )$ allows to estimate
\begin{align}
\notag & \prod_{j =1}^n \lvert \exp ( i \langle\xi_j, y\rangle ) - 1 \rvert  \leq 2^n \prod_{j =1}^n (\|y\| \|\xi_j\| \wedge 1) \leq 2^n \prod_{j =1}^n (( \|y\| \|\xi_j\| )^{\gamma} \wedge 1) \leq 2^n \|y\|^{\gamma n} \prod_{j =1}^n \|\xi_j\|^\gamma,
\end{align}
where we have used the fact that $\lvert e^{ix} - 1\rvert \leq \lvert x\rvert \wedge 2 \leq 2 ( \lvert x\rvert \wedge 1)$, for all $x$. 
 Therefore, 
\begin{align}
\notag & \left\lvert  \E \big(L (x+y,[s,t])-L(x,[s,t])\big)^{n} \right\rvert  \\ 
\notag \leq  & \pi^{-n} \|y\|^{\gamma n} \int_{(\R^d)^n} \int_{[s, t]^n} \prod_{j =1}^n \|\xi_j\|^\gamma  \left\lvert \E \exp\left( - i \sum_{j = 1}^n \langle\xi_j, X_{v_j}\rangle \right)\right\rvert  dv d\xi.  
\end{align} 
Now, Proposition~\ref{prop:int_fourier} with $\eta = \gamma$ yields: 
\begin{align}
\notag & \left\lvert \E \big(L (x+y,[s,t])-L(x,[s,t])\big)^{n} \right\rvert  \leq C^n \|y\|^{\gamma n} (t - s)^{(1 - \alpha(d + \gamma) )n} n^{ n \alpha( d + \gamma+\theta) },
\end{align} 
where $C > 0$ is a function of $\alpha$, $d$, and $\gamma$. This establishes \eqref{eq:mom4}.

Similarly, by~\eqref{eq:local_rep}, using $L(x, s) \leq L(x, t)$ for $0 \leq  s < t$, we get
\begin{eqnarray*}
 &  & \E \lvert L(x,t)-L(x,s)\rvert^n \\
& = & \left((2\pi)^{-n}\int_{[s,t]^n}\int_{\R^n}\exp\left(i\langle x,\sum_{j=1}^n \xi_j\rangle\right)\E\exp\left(-i\sum_{j=1}^n \langle\xi_j, X_{u_j}\rangle\right)d\xi du\right) \\
&\leq& (2\pi)^{-n}\int_{[s,t]^n}\int_{\R^n}\left\vert\E\exp\left(-i\sum_{j=1}^n \langle\xi_j, X_{u_j}\rangle\right)\right\vert d\xi du\\
&\leq& C^n (t - s)^{ (1 - \alpha d) n} n^{ n \alpha (d+\theta) },
\end{eqnarray*}
where the last inequality follows from Proposition~\ref{prop:int_fourier} with $\eta = 0$, and  $C > 0$ is a function of $\alpha$, $d$, and $\gamma$.
\end{proof}
The joint H\"older continuity follows immediately from the Kolmogorov criterion (see e.g.~\cite[Theorem 3.23]{ Kallenberg-2002}).
 \begin{corollary}\label{cor:holder}
 Almost surely, the local time $L(x,t)$ is jointly H\"older continuous in $t$ and $x$.
 \end{corollary}
The next theorem is a modification of Proposition \ref{prop:moment-bounds}, where one shifts the process in the $x$-direction by the value $X_a$, where $a$ is a fixed point. The proof differs from \cite{knsv} as in our general setting, we do not have stationarity of the increments at our disposal.
\begin{theorem}
\label{thm:moment}
Let $a > 0$ and $x\in \rd$ be fixed. Then,
\begin{equation}\label{eq:mom3-shift}
\E \lvert L(x+X_a,t)-L(x+X_a,s)\rvert ^n \leq C^n (t - s)^{(1 - \alpha d) n} n^{ n \alpha (d+\theta) }.
\end{equation}
Moreover, for every $0\leq \gamma<\min\left(1,\frac{1-\alpha d}{2\alpha}\right)$ and $y\in \rd$,
\begin{equation}\label{eq:mom4-shift}
\left\lvert  \E (L(x+y+X_a,[s,t])-L(x+X_a,[s,t]))^{n} \right\rvert  \leq C^n \|y\|^{\gamma n} (t - s)^{(1 - \alpha(d + \gamma) )n} n^{ n \alpha( d + \gamma+\theta) }.
\end{equation}
In both cases the constant $C > 0$ depends only on $\gamma$, $d$, and $\alpha$.
\end{theorem}

\begin{proof}
Let $Y_t = X_t - X_a$. The occupation measure of $Y$ is just the occupation measure of $X_t$ translated by the (random) constant $X_a$. Since the occupation measure of $X_t$ has a continuous density, the occupation measure of $Y_t$ has also a continuous density given by
$L_Y(t,x) = L_X(t,x+X_a)$. Thus, in order to prove the claim, it suffices to show the estimates for $L_Y(t,x)$. For the first claim, we then proceed as before, noting that, again, $L_Y(x, s) \leq L_Y(x, t)$,
\begin{eqnarray*}
&&\E \lvert L_Y(x,t)-L_Y(x,s)\rvert ^n \\&= &(2\pi)^{-n}\int_{[s,t]^n}\int_{(\R^d)^n}\exp\left(i\langle x,\sum_{j=1}^n \xi_j\rangle\right)\E\exp\left(i\sum_{j=1}^n \langle\xi_j, Y_{u_j}\rangle\right)d\xi du\\
&\leq& (2\pi)^{-n} n!\int_{s<u_1<\ldots<u_n<t}\int_{(\R^d)^n}\left\vert\E\exp\left(i\sum_{j=1}^n \langle y_j, Y_{u_j}-Y_{u_{j-1}}\rangle\right)\right\vert dydu.
\end{eqnarray*}
The claim follows from this with the same arguments as the proof of Proposition \ref{prop:moment-bounds} by observing $Y_{u_j}-Y_{u_{j-1}} = X_{u_j}-X_{u_{j-1}}$. The other claim can be proved similarly, and we omit the details.
\end{proof}
The moment bounds obtained above translate  into the following tail estimates by Chebychev's inequality. The proof is rather standard, and we omit the details.
\begin{corollary}\label{cor:tail1} {\rm (i)}\quad  For any finite closed interval $I\subset (0, \infty)$, 
\begin{equation}
\label{eq:tail-L0}
\mathbb{P}(L(x,I)\geq \lvert I\rvert ^{1-\alpha d}u^{\alpha(d+\theta)}) \leq C_1\exp(-c_1u)
\end{equation}
and
\begin{equation}
\label{eq:tail-inc0}
\mathbb{P}(\lvert L(x,I)-L(y,I)\rvert \geq  \lvert I\rvert ^{1-\alpha(d+\gamma)}\| x-y\|^{\gamma}u^{\alpha(d+\gamma+\theta)}) \leq C_2\exp(-c_2u).
\end{equation}

\medskip

\noindent {\rm (ii)}\quad Set $I = [a,a+r]$. Then,
\begin{equation}
\label{eq:tail-L}
\mathbb{P}(L(x+X_a,I)\geq r^{1-\alpha d}u^{\alpha(d+\theta)}) \leq C_1\exp(-c_1u)
\end{equation}
and
\begin{equation}
\label{eq:tail-inc}
\mathbb{P}(\lvert L(x+X_a,I)-L(y+X_a,I)\rvert \geq  r^{1-\alpha(d+\gamma)}\lvert x-y\rvert ^{\gamma}u^{\alpha(d+\gamma+\theta)}) \leq C_2\exp(-c_2u).
\end{equation}
\textcolor{black}{Here all the constants $c_1,c_2,C_1$, and $C_2$ depend only on $\theta$, $\gamma$, $d$, and $\alpha$.}
\end{corollary}
We also recall the following Garsia-Rodemich-Rumsey inequality from \cite{Garsia-Rodemich-Rumsey-1970}.
\begin{proposition} Let $\Psi(u)$ be a non-negative even function on $(- \infty, \infty)$ and $p(u)$ be a non-negative even function on $[-1,1]$. Assume both $p(u)$ and $\Psi(u)$ are non-decreasing for $u \geq 0$. Let $f(x)$ be continuous on $[0, 1]$ and suppose that
\begin{align}
\notag \int_0^1 \int_0^1 \Psi\left( \frac{f(x) - f(y) }{p(x - y)} \right)dx dy \leq B < \infty.
\end{align}
Then, for all $s, t \in [0, 1]$, 
\begin{align}
\notag \lvert f(s) - f(t) \rvert  \leq 8 \int_0^{\lvert s - t\rvert } \Psi^{-1} \left( \frac{4B}{u^2} \right) dp(u).
\end{align}
\end{proposition}
\begin{proposition}
\label{prop:sup-tail}
Suppose that $X$ satisfies Assumption (A2). Then there exists a positive random constant $B$ with finite first moment such that
$$
\sup_{\lvert t-s\rvert <h}\|X_t-X_s\|\leq Bh^{\alpha}(\log 1/h)^{\iota}.
$$
\end{proposition}

\begin{proof}
For simplicity we consider only the case $d=1$. The general case follows by applying the one-dimensional case component-wise.
We apply the Garsia-Rodemich-Rumsey inequality with $p(u) = u^{\alpha}$ and $\Psi(x) = \exp\left(\beta \lvert x\rvert ^{\frac{1}{\iota}}\right)$ for suitably chosen small $\beta>0$ to be determined later. It follows that 
\begin{align}
\label{eq:GRS-fundamental}
\lvert X_t - X_s \rvert  \leq C \int_0^{\lvert s - t\rvert } \beta^{-\iota}\left(\log\left(\frac{4B}{u^2}\right)\right)^{\iota}u^{\alpha-1}du,
\end{align}
where 
$$
B = \int_0^1\int_0^1 \exp\left(\beta\left(\frac{\lvert X_t-X_s\rvert }{\lvert t-s\rvert ^{\alpha}}\right)^{\frac{1}{\iota}}\right)dsdt.
$$
Assuming that $B$ is a finite random variable, it follows from \eqref{eq:GRS-fundamental} that 
\begin{align*}
\lvert X_t - X_s \rvert  &\leq C \int_0^{\lvert s - t\rvert } \beta^{-\iota}\left(\log\left(\frac{4B}{u^2}\right)\right)^{\iota}u^{\alpha-1}du \\
& \leq C(\beta,\iota) \left[\log(\max(1,4B))+1\right]^\iota \int_0^{\lvert t-s\rvert } \left(\log (1/u)\right)^{\iota}u^{\alpha-1}du \\
&\leq C(\beta,\iota)\left[\max(1,B)+1\right]^\iota \lvert t-s\rvert ^{\alpha}\int_0^1 \left(\log (1/\lvert t-s\rvert v)\right)^{\iota}v^{\alpha-1}dv \\
&\leq C(\beta,\iota)\left[\max(1,B)+1\right]\lvert t-s\rvert ^{\alpha}\left[-\log\lvert t-s\rvert \right]^{\iota} \int_0^1 \left(\log v^{-1}\right)^{\iota} v^{\alpha-1}dv \\
&\leq C(\beta,\iota)\left[B+1\right]\lvert  t-s\rvert ^{\alpha} \left[-\log\lvert t-s\rvert \right]^{\iota},
\end{align*}
where we have used $[\max(1,B)+1]^{\iota} \leq \max(1,B)+1$ as $\iota \leq 1$.
Hence it suffices to prove that $B$ is a finite random variable almost surely. For this, by taking expectation we see
\begin{align*}
\E B &=  \sum_{k=1}^\revnew{\infty} \int_0^1 \int_0^1 \frac{\beta^k \E\lvert X_r-X_v\rvert ^{k/\iota}}{k!\lvert r-v\rvert ^{k \alpha / \iota}}drdv\\
&\leq \sum_{k=1}^\revnew{\infty} \frac{1}{k!}(C\beta)^k \left(\frac{k}{\iota}\right)^{\frac{k\iota}{\iota}} \leq \sum_{k=1}^\revnew{\infty} (C\beta)^k \frac{k^k}{k!} \leq C \sum_{k=1}^\revnew{\infty} (C\beta)^k ,
\end{align*}
which is finite for small enough $\beta>0$, and where we have used Stirling's approximation on the last inequality. This completes the proof.
\end{proof}
We are now ready to prove our main theorem. The proof is based on a standard chaining argument, see e.g. \cite[Theorem 4.3]{xiao-ptrf} or \cite[Theorem 1.4]{knsv}. For the reader's convenience however, we present the main arguments in order to keep track on essential exponents that are relatively different in our generalized case.

\begin{proof}[Proof of Theorem \ref{main}]
The proof will be divided into five steps. The first four steps are devoted to the proof of \eqref{eq:main-fixed-s}, while the proof of \eqref{eq:main-sup-s} is presented in step 5. \textcolor{black}{Moreover, we only consider the interval $[s,s+r]$ for notational simplicity, as the part $[s-r,s]$ can be treated with a symmetric argument.}

In the sequel, we denote $g(r) = r^{1-\alpha d}(\log \log r^{-1})^{\alpha(d+\theta)}$, where $r<e$, and we define $C_n = [s,s+2^{-n}]$. With this notation, it suffices to prove that
$$
\limsup_{n\to\infty} \frac{L^*(C_n)}{g(2^{-n})} \leq C
$$ 
almost surely, where $L^*(C_n) = \sup\{L(x,C_n):x\in \overline{X(C_n)}\}$. Throughout this proof, we will denote by $c_i,i=1,2,\ldots$ generic unimportant constants that may change from line to line.\\
\textbf{Step 1:}
From Proposition \ref{prop:sup-tail} we get
$$
\sup_{t\in C_n}\| X_t-X_s\| \leq B2^{-n\alpha}(\log2^n)^{\iota}
$$
for some positive random variable $B$ having finite first moment. It follows easily from Chebychev's inequality and Borel-Cantelli that
$$
\sup_{t\in C_n}\| X_t-X_s\| \leq n^2 2^{-n\alpha} n^{\iota}
$$
for $n\geq n_1(\omega)$. \\
\textbf{Step 2:}
Set $\theta_n = 2^{-n\alpha }(\log \log 2^n)^{-\alpha}$ and define
$$
G_n =\{x \in \rd : \| x\| \leq 2^{-n\alpha}n^{2 +\iota}, x= \theta_n p, \text{ for some }p \in \Z^d\}.
$$
Then 
$$
\# G_n \leq \textcolor{black}{c_1} (\log n)^{\alpha d}n^{(2 +\iota)d}
$$
and \eqref{eq:tail-L} implies
$$
\mathbb{P}\Big( \max_{x\in G_n} L(x+X_s,C_n) \geq \textcolor{black}{c_2} g(2^{-n}) \Big) \leq \textcolor{black}{c_3}  (\log n)^{\alpha d}n^{(2+\iota)d-\textcolor{black}{c_4}},
$$
which is summable by choosing $\textcolor{black}{c_2}$ large enough which in turn corresponds to $\textcolor{black}{c_4}$ being large. Thus, by Borel-Cantelli for large enough $n\geq n_2(\omega)\geq n_1(\omega)$ we have 
\begin{equation}
\label{eq:max-L-bounded}
\max_{x\in G_n} L(x+X_s,C_n) \leq \textcolor{black}{c_2} g(2^{-n}).
\end{equation}
\textbf{Step 3:}
For given integers $n,k\geq 1$ and $x\in G_n$ we set
$$
F(n,k,x) = \{ y = x + \theta_n \sum_{j=1}^k \varepsilon_j2^{-j} \;:\; \varepsilon_j \in \{0,1\}^d, 1\leq j \leq k\}.
$$
Then we say that a pair of two points $y_1,y_2 \in F(n,k,x)$ is linked if $y_2-y_1 = \theta_n \varepsilon 2^{-k}$ for $\varepsilon \in\{0,1\}^d$. Next we fix $0<\gamma< \min\left(1,\frac{1-\alpha d}{2\alpha}\right)$ and choose $\delta \in (0, \infty)$  with
$\delta \alpha(d+\theta+\gamma) < \gamma$. Further set
\begin{align*}
B_n &= \bigcup_{x\in G_n}\bigcup_{k=1}^\infty \bigcup_{y_1,y_2} \{\lvert L(y_1+X_s,C_n)-L(y_2+X_s,C_n)\rvert  \\
&\geq 2^{-n(1-\alpha(d+\theta))}\|y_1-y_2\|^{\gamma}(\textcolor{black}{c_5}2^{\delta k}\log n)^{\alpha(d+\gamma+\theta)}\} ,
\end{align*}
where $\cup_{y_1,y_2}$ is the union over all linked pairs $y_1,y_2\in F(n,k,x)$. Now \eqref{eq:tail-inc} with $u = \textcolor{black}{c_5}2^{\delta k}\log n$ gives
$$
\mathbb{P}(B_n) \leq \textcolor{black}{c_1}(\log n)^{\alpha d}n^{(2+\iota)d}\sum_{k=1}^\infty 4^{kd}\exp\left(-\textcolor{black}{c_6}2^{\delta k}\log n\right),
$$
where we have used the fact $\# G_n \leq \textcolor{black}{c_1}(\log n)^{\alpha d}n^{(2+\iota)d}$ and that for given $k$ there exists less than $4^{kd}$ linked pairs $y_1,y_2$. Now, again by choosing $\textcolor{black}{c_5}$ large enough which corresponds in $\textcolor{black}{c_6}$ to be large, we get 
$$
\sum_{n=2}^\infty(\log n)^{\alpha d}n^{(2+\iota)d}\sum_{k=1}^\infty 4^{kd}\exp\left(-\textcolor{black}{c_6}2^{\delta k}\log n\right) < \infty.
$$
Indeed, by taking $\textcolor{black}{c_6}$ to be large enough, we observe that $n^{-\textcolor{black}{c_6}2^{\delta k}} \leq n^{-pk}$ for every $k\geq 1$ and some arbitrary number $p$, yielding for all $n$ with $4^d < n^p$ that
\begin{align*}
&\sum_{k=1}^\infty 4^{kd}\exp\left(-c_82^{\delta k}\log n\right) \leq \sum_{k=1}^\infty 4^{kd}n^{-pk} \leq \frac{4^{d}}{n^p-4^d},
\end{align*}
and this compensates the growth of $(\log n)^{\alpha d}n^{(2+\iota)d}$ as $p$ is arbitrary.
This further implies, again using Borel-Cantelli, that $B_n$ occurs only finitely many times.

\noindent\textbf{Step 4:}
Let $n$ be fixed and assume that $y \in \rd$ satisfies $\| y\|  \leq 2\cdot 2^{-n\alpha} n^{2+\iota}$. Then we may represent $y$ as $y = \lim_{k\to\infty}  y_k$ with
$$
y_k = x + \theta_n \sum_{j=1}^k \varepsilon_j 2^{-j},
$$
where $y_0= x \in G_n$ and $\varepsilon_j \in \{0,1\}^d$. 
On the event $B_n^c$ we have 
\begin{equation*}
\begin{split}
\lvert L(x+X_s,C_n)-L(y+X_s,C_n)\rvert  & \leq \sum_{k=1}^\infty \lvert L(y_k+X_s,C_n)-L(y_{k-1}+X_s,C_n)\rvert  \\
&\leq \sum_{k=1}^\infty 2^{-n(1-\alpha(d+\gamma))}\|y_k-y_{k-1}\|^{\gamma}(\textcolor{black}{c_5}2^{\delta k}\log n)^{\alpha(d+\theta+\gamma)} \\
&\leq C(d)\sum_{k=1}^\infty 2^{-n(1-\alpha(d+\gamma))}\theta_n^{\gamma}2^{-k\gamma}(\textcolor{black}{c_5}2^{\delta k}\log n)^{\alpha(d+\theta+\gamma)} \\
& \leq \textcolor{black}{c_7} 2^{-n(1-\alpha d)} \sum_{k=1}^\infty (\log \log 2^n)^{-\gamma \alpha}2^{-k\gamma}(\textcolor{black}{c_5}2^{\delta k}\log n)^{\alpha(d+\theta+\gamma)} \\
& \leq \textcolor{black}{c_{8}} 2^{-n(1-\alpha d)} (\log \log 2^n)^{\alpha(d+\theta)} \sum_{k=1}^\infty 2^{(\delta \alpha(d+\theta+\gamma) - \gamma)k} \\
& \leq \textcolor{black}{c_{9}}g(2^{-n}),
\end{split}
\end{equation*}
where the last inequality follows from 
$\delta \alpha(d+\theta+\gamma) < \gamma$. Combining this with \eqref{eq:max-L-bounded} then yields
$$
\sup_{\|\textcolor{black}{y}\|\leq 2^{1-n\alpha} n^{2+\iota}} L(\textcolor{black}{y}+X_s,C_n) \leq \textcolor{black}{c_{10}} g(2^{-n})
$$
or in other words,
$$
\sup_{\|\textcolor{black}{y}-X_s\|\leq 2^{1-n\alpha} n^{2+\iota}} L(\textcolor{black}{y},C_n) \leq \textcolor{black}{c_{10}}g(2^{-n}).
$$
Claim \eqref{eq:main-fixed-s} now follows from Step 1 and the fact $L^*(C_n) = \sup\{L(\textcolor{black}{y},C_n):\textcolor{black}{y}\in \overline{X(C_n)}\}$.  

\noindent\textbf{Step 5:}
It remains to prove \eqref{eq:main-sup-s}. However, this follows exactly by modifying the arguments in Steps 1-4 and following, e.g. \cite[pages 511-512]{knsv}. For this reason we omit the details.
\end{proof}
Proofs of corollaries \ref{cm1} and \ref{cm2} follow directly from Theorem \ref{main} and the arguments presented in \cite{xiao-ptrf} or in \cite{knsv}. Hence we omit the proofs and leave the details for an interested reader.
\subsection{Proofs related to examples}

In order to prove Theorem \Ref{main2} we first recall some basic definitions and further tools from Malliavin calculus, which are needed for the proof. For a detailed introduction and extensive treatment of Malliavin calculus, the reader is referred to the textbook \cite{nualart}.

\subsubsection{Malliavin calculus with $d$-dimensional fractional Brownian motion}
Denote by $R(s,t)$, $s,t \in [0,T],$ the covariance function of the components of $B$, i.e.
 \[
 R(t,s):= \frac12 \Big(t^{2H} + s^{2H} - \lvert t-s\rvert ^{2H}\Big)
 \]
for every $s,t \in [0,T]$. We take $\mathcal{H}$ to be the Hilbert space given by the closure of $\rd$-valued step functions on $[0,T]$ with respect to the scalar product
$$ \langle \Big( \mathbf{1}_{[0,t_1]} , \ldots,  \mathbf{1}_{[0,t_d]} \Big) , \Big( \mathbf{1}_{[0,s_1]} , \ldots,  \mathbf{1}_{[0,s_d]} \Big) \rangle_{\mathcal{H}} = \sum_{l=1}^d R(t_l,s_l) $$
for $t_l, s_l \in [0,T]$, $1 \leq l \leq d$, and $\mathcal{H}$ is equipped with the corresponding norm $\| \cdot \|_{\mathcal{H}}$. The Wiener integral with underlying $d$-dimensional fractional Brownian motion $B$ is defined by
$$B(h) = \int_0^1 \langle h_s, dB_s \rangle$$
for every $h \in \mathcal{H}$ and satisfies
$$ \mathbb{E} [ B(h_1) B(h_2) ] = \langle h_1, h_2 \rangle_{\mathcal{H}}$$
for all $h_1, h_2 \in \mathcal{H}$. Moreover, the set of cylindrical random variables denoted by $\mathcal{S}$ is defined as the set of all measurable real-valued random variables $F$ with the property that
$$ F = f\big( B(h_1), \ldots, B(h_n) \big),$$
where $n \in\N$, $h_1, \ldots, h_n \in \mathcal{H}$ and $f \colon \rr^n \to \rr$ is a smooth function with bounded derivatives of all orders. Then for every such $F \in \mathcal{S}$ the Malliavin derivative (with respect to $B$) is defined as the stochastic process $(\mathbf{D}_t F)_{t\in[0,T]}$ with values in $\mathcal{H}$ given by
$$\mathbf{D}_t F = \sum_{i=1}^n h_i (t) \frac{\partial f}{\partial x_i} \big( B(h_1), \ldots, B(h_n) \big), \quad t\in [0,T]. $$
For our purposes we also need the introduction of iterated derivatives $\mathbf{D}^k_{t_1, \ldots, t_k} F =\mathbf{D}_{t_1} \ldots \mathbf{D}_{t_k} F$ for $k \in \N$, $t_1, \ldots, t_k \in [0,T]$. Thus, $\mathbf{D}^k_{t_1, \ldots, t_k} F$ is a random variable attaining values in $\mathcal{H}^{\otimes k}$, which denotes the $k$-fold tensor product and is equipped with the corresponding norm $\| \cdot \|_{\mathcal{H}^{\otimes k}}$. For every $p \in [1, \infty)$ one defines $\mathbb{D}^{k,p}$ as the closure of the set of cylindrical random variables with respect to the seminorm
$$ \| F\|_{k,p} = \Big( \mathbb{E} [\lvert F\rvert ^p] + \sum_{j=1}^k \mathbb{E} [ \| \mathbf{D}^j F \|^p_{\mathcal{H}^{\otimes j}} ] \Big)^{\frac{1}{p}}$$
and
$$ \mathbb{D}^\infty = \bigcap_{p \in \N} \bigcap_{k\in \N} \mathbb{D}^{k,p}.$$
For a random vector $F=(F^1, \ldots, F^d)$ with components in $ \mathbb{D}^\infty$, its Malliavin matrix is defined as
$$\gamma_F = \big( \langle  \mathbf{D} F^i,  \mathbf{D} F^j \rangle_{\mathcal{H}} \big)_{i,j=1, \ldots, d}$$
and $F$ is called non-degenerate if $\gamma_F$ is almost surely invertible with
$$ (\operatorname{det} \gamma_F )^{-1} \in \bigcap_{p \in \N} L^p (\Omega).$$

Recall that, with $W$ denoting the underlying standard Brownian motion in the kernel representation of $B$ (see \cite[Section 5.1.3]{nualart} for details), it is well-known that $B$ and $W$ generate the same filtration. For our purposes, we will also perform Malliavin calculus with respect to the underlying $W$ as well. In order to distinguish corresponding notation we denote by $D$ the Malliavin derivative with respect to $W$ and, similarly we carry over the notations for the sets ${D}^{k,p}$ and $D^\infty$.
 For the relation between $D$ and $\mathbf{D}$ we refer to \cite[Section 5.2.1]{nualart}. 

In our approach we will be led to work with conditional Malliavin matrices. To this end, we denote $L_t^2 = L^2 ([t,T])$ and $\mathbb{E}_t = \mathbb{E} [ \cdot \mid \mathcal{F}_t]$, $t \in [0,T]$ the conditional expectation with respect to $\mathcal{F}_t$. For a 
random variable $F$ and $t \in [0,T]$, $k \in \N_0$, $p \in (0, \infty)$ we then define the conditional seminorm
$$ \| F\|_{k,p;t} = \Big( \mathbb{E}_t [\lvert F\rvert ^p] + \sum_{j=1}^k \mathbb{E}_t [ \| D^j F \|^p_{(L_t^2)^{\otimes j}} ] \Big)^{\frac{1}{p}} .$$
Furthermore, we define the conditional Malliavin matrix of $F$ as
$$\Gamma_{F,t} = \big( \langle  DF^i, D F^j \rangle_{L_t^2} \big)_{i,j=1, \ldots, d}$$
for $t \in [0,T]$. With this notation in hand, the next result is a restatement of \cite[Proposition 5.6]{baudoin} (see also \cite[Proposition 2.2]{lou}) and the conditional formulation of a version in \cite[Proposition 2.1.4]{nualart}.

\begin{lemma}\label{ip}
Let $k,n \in \N$, $F=(F^1, \ldots, F^d)$ a non-degenerate random vector and $G$ a random variable. Suppose that $F^1, \ldots, F^d, G \in D^\infty$ and 
$$ (\operatorname{det} \Gamma_{F,s} )^{-1} \in \bigcap_{p \in \N} L^p (\Omega), \quad s \in [0,T].$$
Then for every multi-index $\alpha \in \{ 1, \ldots, d\}^k$ there exists $H_\alpha^s (F,G) \in D^\infty$ such that
$$\mathbb{E}_s \Big[ (\partial_\alpha \varphi) (F) G \Big] = \mathbb{E}_s \Big[ \varphi (F) H_\alpha^s (F,G) \Big]$$
for every smooth function $\varphi: \rd \to \rr$ such that all of its partial derivatives have at most polynomial growth. Moreover, $H_\alpha^s (F,G) $ is recursively defined by
\begin{align*}
H_{(i)}^s (F,G) &= \sum_{j=1}^n \delta_s \Big( G ( \Gamma_{F,s}^{-1})_{i,j} DF^j \Big) , \\
H_{\alpha}^s (F,G) &= H_{(\alpha_k)}^s \Big( F,H_{(\alpha_1, \ldots, \alpha_{k-1})}^s (F,G)\Big) ,
\end{align*}
with $\delta_s$ denoting the Skorokhod integral with respect to $W$ on the interval $[s,T]$ (see \cite[Section 1.3.2]{nualart}). Additionally, for $1 \leq p < q < \infty$ with $\frac{1}{p} = \frac{1}{q} + \frac{1}{r}$ we have
$$ \| H_{\alpha}^s (F,G) \|_{p;s} \leq c \| \Gamma_{F,s}^{-1} DF \|_{k,2^{k-1}r;s}^k \| G\|_{k,q;s},$$
where $c \in (0, \infty)$ is a constant that only depends on $p$ and $q$. In particular, there exists a constant $C \in (0, \infty)$ such that
\begin{equation}
\label{eq:H-norm}
 \| H_{\alpha}^s (F,G) \|_{p;s} \leq C \| \operatorname{det} \Gamma_{F,s}^{-1} \|_{k,2^{\textcolor{black}{k}}r;s}^k  \|  DF \|_{k,2^{\textcolor{black}{k}}r;s}^k \| G\|_{k,q;s}.
 \end{equation}
\end{lemma}

Furthermore, we will employ the following result which is proven in \cite[Proposition 5.9, equation (26) and (27)]{baudoin}. 

{\color{black}
\begin{lemma}\label{est}
Let $\varepsilon \in (0,T)$, $H \in (\frac{1}{4},1)$ and $(X_t)_{t\in[0,T]}$ as in Theorem \ref{main2}. Then there exists a constant $c=c(\varepsilon) \in (0, \infty)$ and $r>0$ such that, for $\varepsilon\leq s \leq t \leq T$, it holds
\begin{equation} 
\label{eq:inverse-inc}
 \| \operatorname{det} \Gamma_{X_t-X_s,s}^{-1} \|_{k,2^{k+2};s} \leq \frac{c^k}{(t-s)^{2H}} \Big( \mathbb{E}_s [1+\mathcal{M}^{r2^{k+2}} ] \Big) ^{2^{-k-2}} ,
 \end{equation}
\begin{equation} 
\label{eq:inverse-initial}
\| \operatorname{det} \Gamma_{X_t-x}^{-1} \|_{k,2^{k+2}}  \leq \frac{c^k}{t^{2H}}, \quad k \geq1,
\end{equation}
\begin{equation} 
\label{eq:derivative-inc}
\| D(X_t - X_s) \|_{k,2^{k+2};s} \leq c^k (t-s)^{H} \Big( \mathbb{E}_s [1+\mathcal{M}^{r2^{k+2}} ] \Big) ^{2^{-k-2}},
\end{equation}
and
\begin{equation} 
\label{eq:derivative-initial}
 \| D(X_t -x) \|_{k,2^{k+2}} \leq c^k t^{H} , \quad k\geq 1,
\end{equation}
where $k \in \N$ and $\mathcal{M} \in D^\infty$ is a random variable with positive values and Gaussian tails.
\end{lemma}
}
\textcolor{black}{
\begin{remark}
By carefully examining the proof of \cite[Proposition 5.9]{baudoin}, based on the method introduced in \cite{inahama}, one notes that in the statement \cite[Proposition 5.9]{baudoin} the power $2^{k+2}$ is missing in $\mathcal{M}$, while it is clearly present in the proof. While this makes the upper bounds slightly worse, it does not play any significant role in our estimates. The reason is that, since $\mathcal{M}$ has Gaussian tails, the moment of order $2^{k}$ can be upper bounded by $C^k$. Note also that above we have formulated \eqref{eq:inverse-initial} and \eqref{eq:derivative-initial} slightly different compared to \cite[Lemma 4.1 and Eq. (27)]{baudoin}, as \cite{baudoin} used general $p$ for the norm $\Vert \cdot\Vert_{k,p}$ while in our formulation we use $p = 2^{k+2}$ corresponding to our choice. This also reveals clearer the connection to equations \eqref{eq:inverse-inc} and \eqref{eq:derivative-inc}. Indeed, the proofs of equations \eqref{eq:inverse-inc}-\eqref{eq:derivative-initial} are similar, and actually one could add the term $( \mathbb{E} [1+\mathcal{M}^{r2^{k+2}} ] ) ^{2^{-k-2}}$ in \eqref{eq:inverse-initial} and \eqref{eq:derivative-initial} (note that this term is currently included in the constant $c^k$), see e.g. \cite[Lemma 4.1 and its proof]{baudoin}. 
\end{remark}
}
\begin{remark}
The random variable $\mathcal{M}$ in Lemma \ref{est} can be specified more explicitly. Indeed, it is derived from the Besov norm of a lift of $B$ as a rough path, see \cite[p. 2583]{baudoin} for details.
\end{remark}

We will also need the following auxiliary result in our proof.
\begin{lemma}
\label{lemma:recursion}
Suppose that a sequence $\alpha_h(k)$, $h,k\geq 1$ satisfies 
$\alpha_k(k)=1=\alpha_1(k)$ for all $k\geq 1$, $\alpha_h(k) = 0$ for $h> k$, and for $2\leq h\leq k$ 
$$
\alpha_h(k+1)=h\alpha_h(k) + \alpha_{h-1}(k).
$$
Then for all $h,k\geq 1$ we have 
$$
\alpha_h(k) \leq C^k k^k.
$$
\end{lemma}

\begin{proof}
Without loss of generality we can assume $k\geq 3$. From the recursion we get 
$$
\alpha_j(k+1)-\alpha_{j-1}(k) = j\alpha_j(k).
$$
In particular, plugging $j=k-p$ for an integer $p\geq 0$ such that $k-p-1\geq 1$, leads to 
$$
\alpha_{k-p}(k+1) - \alpha_{k-p-1}(k) = (k-p)\alpha_{k-p}(k).
$$
Denoting $\beta_p(k) = \alpha_{k-p-1}(k)$, we observe that this is equivalent to
$$
\beta_p(k+1)-\beta_p(k) = (k-p)\alpha_{k-p}(k).
$$
Summing over $k=p+2,p+3,\ldots, K-1$, where $K\geq 3$, and using $\beta_p(p+2) = \alpha_{1}(p+2)=1$ hence gives
\begin{align*}
\alpha_{K-p-1}(K) - 1 &= \beta_{p}(K) - \beta_p(p+2)\\
&= \sum_{k=p+2}^{K-1} \left[\beta_p(k+1)-\beta_p(k)\right] \\
&= \sum_{k=p+2}^{K-1} (k-p)\alpha_{k-p}(k).
\end{align*}
This leads to 
$$
\alpha_{K-p-1}(K) = 1 + \sum_{k=p+2}^{K-1} (k-p)\alpha_{k-p}(k) = \sum_{k=p+1}^{K-1} (k-p)\alpha_{k-p}(k).
$$
From this it follows by induction \textcolor{black}{in $p$ that, for all fixed $K \geq 3$, we have}
$$
\alpha_{K-p-1}(K) \leq \frac{K^{2+2p}}{(2p+2)!!},
$$
where $(2p+2)!! = 2^{p+1}(p+1)!$ is the double factorial. Indeed, for $p=0$ we get directly 
$$
\alpha_{K-1}(K) = \sum_{k=1}^{K-1} k = \frac{(K-1)K}{2} \leq \frac{K^2}{2}
$$
and, assuming that the claim is valid for some $p$, then for $p+1$ we get 
$$
\alpha_{K-p-2}(K) =\sum_{k=\textcolor{black}{p+2}}^{K-1}(k-p-1)\alpha_{k-p-1}(k) \leq \sum_{k=\textcolor{black}{p+2}}^{K-1}\frac{k^{3+2p}}{(2p+2)!!} \leq \frac{K^{4+2p}}{(2p+4)(2p+2)!!} ,
$$
proving the induction step. Plugging in now $j=K-p-1$ leads to 
$$
\alpha_j(K) \leq \frac{K^{2K-2j}}{2^{K-j}(K-j)!}.
$$
Clearly this satisfies the claimed bound when $K-j$ is small. On the other hand, when $K-j$ is large, Stirling's approximation gives 
$$
\frac{K^{2K-2j}}{2^{K-j}(K-j)!} \leq C^K K^{2K-2j}(K-j)^{-(K-j)}.
$$
In particular, \textcolor{black}{by setting $\delta = \frac{j}{K}$} we get
$$
\alpha_j(K) \leq C_\delta^K K^{(1-\delta)K} \leq C^K K^K
$$
\textcolor{black}{where now $C$ can be chosen independently of $\delta$, since $\delta\in[0,1]$. This} yields the result.
\end{proof}
\begin{remark}
We note that the upper bound stated in the lemma is asymptotically essentially sharp. Indeed, from the recursion we get 
$$
\alpha_j(k+1) \geq j\alpha_j(k),
$$
from which it follows that $\alpha_j(k) \geq j^{k-j}$ for all $k\geq j$. \textcolor{black}{Setting now $\delta = j/k$} gives us the lower bound
$$
\alpha_j(k) \geq \delta^{(1-\delta)k}k^{(1-\delta)k}.
$$
As $\delta>0$ is arbitrary \textcolor{black}{(by choosing $j$ suitably)}, we observe that the upper bound $k^k$ cannot be improved.
\end{remark}
Now we are in position to prove Theorem \Ref{main2}.

\begin{proof}[Proof of Theorem \ref{main2}]
Recall that $H \in (0,\frac{1}{d})$ by assumption. We first prove that $(X_t)_{t\in[0,T]}$  satisfies assumption (A1) in Theorem \ref{main} with $\alpha = H$ by using Malliavin calculus. Let $m\in \N$, $\xi_j = (\xi_{j,1}, \ldots, \xi_{j,d}) \in (\rr\setminus \{0\} )^d$, $j=1, \ldots, m$, and $0=t_0 \textcolor{black}{ < \varepsilon\leq }t_1 < \ldots < t_m<T$. \textcolor{black}{Note that, by Remark \ref{remark:away-from-zero}, we obtain results for the local time $L(x,[\varepsilon,t])$ along the proof of Theorem \ref{main} once we have established (A1) for $t_1\geq \varepsilon$ and (A2). As $\varepsilon>0$ is arbitrary, this is sufficient to prove all the claims of Theorem \ref{main2} except the existence of $L(0,[0,t])$ that we will establish at the end of the proof by a limiting argument. As such, in the sequel we will omit the dependence on $\varepsilon>0$ and the reader should keep in mind that at this point we have restricted ourselves away from zero into a closed subinterval $I \subset (0,T)$ (so that $t_1\geq \varepsilon$ for some $\varepsilon>0$).}

\textcolor{black}{In order to establish (A1),} the idea is to apply Lemma \ref{ip} for arbitrary multi-indexes $k_j =(k_{j,1}, \ldots, k_{j,d})$, $1 \leq j \leq m$. In the following we will use the notation $\lvert k_j  \rvert = k_{j,1} + \ldots + k_{j,d}$. \rev{Applying the first identity in Lemma \ref{ip} we write for $j =1$}

\begin{align} \label{661}
\begin{split}
    & \mathbb{E} \Big[ \exp \Big( i \sum_{h=1}^m \langle \xi_h, X_{t_{h}}-X_{t_{h-1}}\rangle \Big) \Big] 
 = \prod_{l=1}^d \frac{1}{(i\xi_{1,l})^{k_{1,l}}}  \mathbb{E} \Big[ \exp \big( i \langle \xi_1, X_{t_{1}}-X_{t_{0}}\rangle \big) H_{k_1}^{t_{0}} (X_{t_{1}}-X_{t_{0}},Y)  \Big] ,
\end{split}
\end{align}
where we set 
$$Y =\exp \Big( i \sum_{h=2}^m \langle \xi_h, X_{t_{h}}-X_{t_{h-1}}\rangle \Big)$$ 
and $ H_{k_1}^{t_0} (X_{t_{1}}-X_{t_{0}},Y)$ as specified in Lemma \ref{ip}. \textcolor{black}{In order to apply bound \eqref{eq:H-norm}, we need to estimate the norm $\Vert Y\Vert_{\lvert k_1\rvert,q,s}$ for some $q>1$ and $s>0$ that we will do next}. Using the notation $k=\lvert k_1\rvert$, $Z_k = X_{t_k} - X_{t_{k-1}}$, $Z=(Z_2,\ldots,Z_m)$ and setting 
$$
Y = f(Z_2,\ldots,Z_m) = \exp \Big( i \sum_{h=2}^m \langle \xi_h, X_{t_{h}}-X_{t_{h-1}}\rangle \Big),
$$ the first derivative is given by 
$$
Df(Z) = \sum_{j=2}^m f'_{x_j}(Z)DZ_j
$$
and the second derivative is given by 
\begin{align*}
D^2 f(Z) &= \sum_{j=2}^m f'_{x_j}(Z) D^2 Z_j \\
&+ \sum_{j_1,j_2=2}^m f''_{x_{j_1}x_{j_2}}(Z) DZ_{j_1} \otimes DZ_{j_2}.
\end{align*}
Similarly, the third derivative is given by 
\begin{align*}
D^3 f(Z) &= \sum_{j=2}^m f'_{x_j}(Z) D^3 Z_j + 3 \sum_{j_1,j_2=2}^m f''_{x_{j_1}x_{j_2}}(Z) D^2Z_{j_1}\otimes DZ_{j_2} \\
&+ \sum_{j_1,j_2,j_3=2}^m f^{(3)}_{x_{j_1}x_{j_2}x_{j_3}}(Z) DZ_{j_1}\otimes DZ_{j_2} \otimes DZ_{j_3}.
\end{align*}
Taking iterated derivatives gives us 
\begin{align*}
D^kf(Z) &= \sum_{j=2}^m f'_{x_j}(Z) D^kZ_j + \sum_{j_1,j_2=2}^m f''_{x_{j_1}x_{j_2}}(Z) A_2^{(k)}(j_1,j_2) + \ldots \\
&+ \sum_{j_1,\ldots,j_{k-1}=2}^m f^{(k-1)}_{x_{j_1}x_{j_2}\ldots x_{j_{k-1}}}(Z) A_{k-1}^{(k)}(j_1,\ldots,j_{k-1}) \\
&+ \sum_{j_1,\ldots,j_k=2}^m f^{(k)}_{x_{j_1}x_{j_2}\ldots x_{j_k}}(Z) A_k^{(k)}(j_1,\ldots,j_k),
\end{align*}
where each $A_h^{(k)}(j_1,\ldots,j_h)$ consists of (tensor) products of iterated derivatives of the form $DZ_{j}$, with a total of $h$ different indices in $Z_j$ and with a total of $k$ derivatives. \rev{Indeed, observe that it holds
\begin{equation}\label{revc1}
A^{(k+1)}_{h} (j_1, \ldots , j_h) = DA^{(k)}_h (j_1, \ldots , j_h) \mathbb{1}_{\{h \leq k\}} + D Z_{j_h} \otimes A^{(k)}_{h-1}(j_1,\ldots , j_{h-1} ) \mathbb{1}_{\{ h \geq 1\} }.
\end{equation}}
\textcolor{black}{
Denote by $e_v$ the amount of derivatives one takes on the term $Z_{j_v}$. Taking into account that 
$$
f^{(h)}_{x_{j_1}x_{j_2}\ldots x_{j_h}}(Z) = i^h f(Z)  \xi_{j_1}\ldots \xi_{j_h}
$$
and by \textcolor{black}{the H\"older inequality for conditional expectations}
we obtain, for each term in $D^kf(Z)$,  an estimate
$$
\mathbb{E}_s [ \| f^{(h)}_{x_{j_1}x_{j_2}\ldots x_{j_h}}(Z) A_{h}^{(k)}(j_1,\ldots,j_{h}) \|^q_{(L_s^2)^{\otimes k}} ] \leq \alpha^q_h(k)\prod_{v=1}^h |\xi_{j_v}|^q \textcolor{black}{\left[\mathbb{E}_s \Vert D^{e_{v}}Z_{j_v}\Vert^{qh}_{(L^2_s)^{\otimes e_v}}\right]^{\frac{1}{h}}},
$$
where $\alpha_h(k)$ corresponds to the amount of terms in $A_h^{(k)}$\footnote{\textcolor{black}{For example for $k=2$, the amount of terms in $A^{(2)}_h$ for $h=1,2$ is only one, while for $k=3$ the amount of terms in $A^{(3)}_h$ is 1 for $h\in\{1,3\}$, while for $h=2$ the amount is 3. This is due to the fact that the mixed derivatives of the form $D^2Z_{j_1}\otimes DZ_{j_2}$ arise by an application of the product rule once from the term $f'_{x_j}(Z) D^2 Z_j$ and two times from the term $f''_{x_{j_1}x_{j_2}}(Z) DZ_{j_1} \otimes DZ_{j_2}$.}}. Here, by \eqref{eq:derivative-inc}\textcolor{black}{, Jensen's inequality for conditional expectations, and since $qh\leq qk \leq 2^{k+2}$ (we will choose $q\leq 2$ later on)}, we have
$$
\textcolor{black}{\mathbb{E}_s \Vert D^{e_{v}}Z_{j_v}\Vert^{qh}_{(L^2_s)^{\otimes e_v}} \leq C^{kqh}(t_{j}-t_{j-1})^{qHh}\left[\E_s[1+\mathcal{M}^{r2^{k+2}}]\right]^{qh2^{-k-2}}.}
$$ 
\textcolor{black}{Since also $h\leq k$,} this leads to the bound 
$$
\mathbb{E}_s [ \| f^{(h)}_{x_{j_1}x_{j_2}\ldots x_{j_h}}(Z) A_{h}^{(k)}(j_1,\ldots,j_{h}) \|^q_{(L_s^2)^{\otimes k}} ] \leq C^{q\textcolor{black}{k^2}}\alpha^q_h(k)c_{s,\omega,k}^{\textcolor{black}{qk}}\prod_{v=1}^h |\xi_{j_v}|^q(t_{j_v}-t_{j_v-1})^{qH},
$$
where 
$$
c_{s,\omega,k} = \left[\E_s[1+\mathcal{M}^{\textcolor{black}{r}2^{k+2}}]\right]^{2^{-k-2}}.
$$ 
For} the coefficients $\alpha_h(k)$ we obtain $\alpha_k(k)=1=\alpha_1(k)$ for all $k\geq 1$, $\alpha_h(k) = 0$ for $h> k$, and for $2\leq h\leq k$ we have  
$$
\alpha_h(k+1)=h\alpha_h(k) + \alpha_{h-1}(k).
$$
Indeed, this can be seen \rev{from the definition of $A^{(k+1)}_{h}$ in \eqref{revc1}}, equivalently by observing that first order derivatives in $f$ arise only from the factor $D^kZ_j$, and $k$:th order derivatives arise only from the single factor $DZ_{j_1}\otimes \ldots \otimes DZ_{j_k}$ giving us $\alpha_k(k)=\alpha_1(k)=1$. Also, $h$:th order derivatives of $f$ for $h>k$ are not present if differentiating only $k$ times, giving us $\alpha_h(k)=0$ for $h>k$. For the values $2\leq h\leq k$, we observe that terms involving the $h$:th order derivative of $f$ when differentiating $k+1$ times arise from the product rule $D(FG) = GDF + FDG$ either 1) from the term $f^{(h)}_{x_{j_1}x_{j_2}\ldots x_{j_h}}(Z)A_{h}(k)$ when the derivative is taken in $A_h(k)$, or 2) from the term $f^{(h-1)}_{x_{j_1}x_{j_2}\ldots x_{j_{h-1}}}(Z)A_{h-1}(k)$ when the derivative is taken in $f^{(h-1)}$. As $A_h(k)$ consists of terms involving exactly $h$ (tensor) products of derivatives of $Z$, the first case contributes a number of terms $h\alpha_h(k)$. The second case on the other hand contributes with $\alpha_{h-1}(k)$, the amount of terms in $A_{h-1}(k)$. Hence, by Lemma \ref{lemma:recursion}, we have 
$$
\alpha_h(k) \leq C^kk^k.
$$
Hence, overall, we obtain the estimate
\begin{align*}
\mathbb{E}_s [ \| D^k Y \|^q_{(L_s^2)^{\otimes k}} ] &\leq \sum_{h=1}^k \sum_{j_1, \ldots, j_h=2}^m \mathbb{E}_s [ \| f^{(h)}_{x_{j_1}x_{j_2}\ldots x_{j_h}}(Z) A_{h}^{(k)}(j_1,\ldots,j_{h})\|^q_{(L_s^2)^{\otimes k}} ] \\
& \leq  c_{s,\omega,k}^{\textcolor{black}{qk}} C^{q\textcolor{black}{k^2}} k^{qk} \sum_{h=1}^k \sum_{j_1, \ldots, j_h=2}^m \prod_{v=1}^h |\xi_{j_v}|^q(t_{j_v}-t_{j_v-1})^{qH}.
\end{align*}
This gives us 
\begin{align*}
\Vert Y\Vert_{k,q,s} &\leq \Bigg[ 1 +  \sum_{v=1}^k \textcolor{black}{c_{s,\omega,k}^{qv}}C^{q\textcolor{black}{v^2}} v^{qv} \sum_{h=1}^v \sum_{j_1, \ldots, j_h=2}^m \prod_{v=1}^h |\xi_{j_v}|^q(t_{j_v}-t_{j_v-1})^{qH} \Bigg]^{\frac{1}{q}} \\
&\leq 1 + c_{s,\omega,k}^{\textcolor{black}{k}} C^{\textcolor{black}{k^2}} k^k \Bigg[\sum_{h=1}^k \sum_{j_1, \ldots, j_h=2}^m \prod_{v=1}^h |\xi_{j_v}|^q(t_{j_v}-t_{j_v-1})^{qH} \Bigg]^{\frac{1}{q}}.
\end{align*}
Now if 
$$  |\xi_{j_v}|(t_{j_v}-t_{j_v-1})^H \leq 1 $$
for every $v \in \{1,\ldots, h\}$ then we can simply estimate 
$$ \Vert Y\Vert_{k,q,s}\leq \textcolor{black}{(1+c_{s,\omega,k}^{\textcolor{black}{k}})} C^{\textcolor{black}{k^2}} k^k m^{k},$$
so we will only consider the case if there is some $v$ such that
$$  |\xi_{j_v}|(t_{j_v}-t_{j_v-1})^H > 1 .$$
Without loss of generality assume that
$$\revnew{ \max_{j =2, \ldots, m} |\xi_{j}|(t_{j}-t_{j-1})^H = |\xi_{2}|(t_{2}-t_{1})^H  }$$
and that
$$\max_{l =1, \ldots, d } |\xi_{2,l}| = |\xi_{2,1}|.$$
Then we obtain the estimate
\begin{equation} \label{771}
    \Vert Y\Vert_{k,q,s} \leq \textcolor{black}{(1+c_{s,\omega,k}^{\textcolor{black}{k}})} C^{\textcolor{black}{k^2}} k^k m^k  |\xi_{2,1}|^k(t_{2}-t_{1})^{kH}.
\end{equation}
\textcolor{black}{
Note that since $\mathcal{M}$ has Gaussian tails, its moments grow similarly as for a Gaussian random variable. In particular, this gives that 
$$
c_{0,\omega,k} = \left[\E[1+\mathcal{M}^{\textcolor{black}{r}2^{k+2}}]\right]^{2^{-k-2}} \leq C^k,
$$
and hence for $s=0$ we have }
$$\| Y\|_{\lvert k_1 \rvert,q} \leq  C^{\textcolor{black}{\lvert k_1 \rvert^2}} \lvert k_1 \rvert^{\lvert k_1 \rvert} m^{\lvert k_1 \rvert} (t_{2}-t_{1})^{\lvert k_1 \rvert H}  {|\xi_{2,1}|^{\lvert k_1 \rvert}} $$ 
for some positive and finite constant $C$. Then, using equation \eqref{661}, estimate \eqref{eq:H-norm} of Lemma \ref{ip} with $(p,r,q,s) = (1,8,8/7,0)$, and then estimates \eqref{eq:inverse-initial} and \eqref{eq:derivative-initial}, we obtain 
\begin{align} \label{55eq1}
\begin{split}
& \Bigg| \mathbb{E} \Big[ \exp \Big( i \sum_{h=1}^m \langle \xi_h, X_{t_{h}}-X_{t_{h-1}}\rangle \Big) \Big] \Bigg|  \\
& \quad \leq C\| \operatorname{det} \Gamma_{X_{t_1}-X_{t_0}}^{-1} \|_{\lvert k_1\rvert,2^{\lvert k_1\rvert+2}}^{\lvert k_1\rvert}  \|  D(X_{t_1}-X_{t_0}) \|_{\lvert k_1\rvert,2^{\lvert k_1\rvert+2}}^{\lvert k_1\rvert} \| Y\|_{\lvert k_1\rvert,q}\prod_{l=1}^d \frac{1}{|\xi_{1,l}|^{k_{1,l}}}\\
& \quad \leq C^{\lvert k_1 \rvert^2} \lvert k_1 \rvert^{\lvert k_1 \rvert} m^{\lvert k_1 \rvert} \Big(  \frac{1}{(t_1-t_{0})^{\lvert k_1\rvert H}}  \Big) (t_{2}-t_{1})^{\lvert k_1 \rvert H}{|\xi_{2,1}|^{\lvert k_1 \rvert}} \prod_{l=1}^d \frac{1}{|\xi_{1,l}|^{k_{1,l}}}    \\
& \quad \leq C^{\lvert k_1 \rvert} \lvert k_1 \rvert^{(1+\delta)\lvert k_1 \rvert} m^{\lvert k_1 \rvert} \Big(  \frac{1}{(t_1-t_{0})^{\lvert k_1\rvert H}}  \Big) (t_{2}-t_{1})^{\lvert k_1 \rvert H}{|\xi_{2,1}|^{\lvert k_1 \rvert}} \prod_{l=1}^d \frac{1}{|\xi_{1,l}|^{k_{1,l}}},
\end{split}
\end{align}
where in the last line we have used $C^{k^2} \leq C_2^k k^{\delta k}$ for any $\delta>0$. 
Observe that \eqref{55eq1} holds for every integer\textcolor{black}{s}  $k_{1,l}$. Replacing $k_{1,l}$ by $k_{1,l}m$ in \eqref{55eq1} we get 
\begin{align} \label{55eq2}
\begin{split}
& \Bigg| \mathbb{E} \Big[ \exp \Big( i \sum_{h=1}^m \langle \xi_h, X_{t_{h}}-X_{t_{h-1}}\rangle \Big) \Big] \Bigg|  \\
& \quad \leq C^{\lvert mk_1 \rvert} \lvert mk_1 \rvert^{\textcolor{black}{(\textcolor{black}{2}+\delta)}\lvert mk_1 \rvert} \Big(  \frac{1}{(t_1-t_{0})^{\lvert m k_1\rvert H}}  \Big) (t_{2}-t_{1})^{\lvert mk_1 \rvert H} {|\xi_{2,1}|^{\lvert m k_1 \rvert}}\prod_{l=1}^d \frac{1}{|\xi_{1,l}|^{mk_{1,l}}}     .
\end{split}
\end{align}
\rev{Now if $j=m$ we proceed similarly and write}
\begin{align*}
& \mathbb{E} \Big[ \exp \Big( i \sum_{h=1}^m \langle \xi_h, X_{t_{h}}-X_{t_{h-1}}\rangle \Big) \Big] \\
& = \prod_{l=1}^d \frac{1}{(i\xi_{m,l})^{k_{m,l}}}  \\
&  \times \small{\mathbb{E} \Big[ \exp \Big( i \sum_{h=1}^{m-1} \langle \xi_h, X_{t_{h}}-X_{t_{h-1}}\rangle \Big)  \mathbb{E}_{t_{m-1}} \big[ \exp \big( i \langle \xi_m, X_{t_{m}}-X_{t_{m-1}}\rangle \big) H_{k_m}^{t_{m}} (X_{t_{m}}-X_{t_{m-1}},1) \big]  \Big] },
\end{align*}
with $ H_{k_m}^{t_m} (X_{t_{m}}-X_{t_{m-1}},1)$ as specified in Lemma \ref{ip}. Proceeding as above but using \eqref{eq:inverse-inc} and \eqref{eq:derivative-inc} instead of \eqref{eq:inverse-initial} and \eqref{eq:derivative-initial} (recall that $t_1\geq \varepsilon$), we obtain
\begin{align} \label{55eq3}
\begin{split}
& \Bigg| \mathbb{E} \Big[ \exp \Big( i \sum_{h=1}^m \langle \xi_h, X_{t_{h}}-X_{t_{h-1}}\rangle \Big) \Big] \Bigg|  \\
& \quad \leq C\mathbb{E} \left[\| \operatorname{det} \Gamma_{X_{t_m}-X_{t_{m-1}}}^{-1} \|_{\lvert k_m\rvert,2^{\lvert k_m\rvert+2},t_{m-1}}^{\lvert k_m\rvert}  \|  D(X_{t_m}-X_{t_{m-1}}) \|_{\lvert k_m\rvert,2^{\lvert k_m\rvert+2},t_{m-1}}^{\lvert k_m\rvert} \prod_{l=1}^d \frac{1}{|\xi_{m,l}|^{k_{m,l}}}\right]\\
&\textcolor{black}{\quad \leq C^{|k_m|^2}  \frac{1}{(t_m-t_{m-1})^{\lvert k_m\rvert H}}   \prod_{l=1}^d \frac{1}{|\xi_{m,l}|^{k_{m,l}}}}\\
&\quad \leq C^{|k_m|}\textcolor{black}{|k_m|^{\delta |k_m|}}  \frac{1}{(t_m-t_{m-1})^{\lvert k_m\rvert H}}   \prod_{l=1}^d \frac{1}{|\xi_{m,l}|^{k_{m,l}}},
\end{split}
\end{align}
where we have again used the fact that the random variable $\mathcal{M}$ has Gaussian tails, and hence, thanks to Jensen's inequality for conditional expectations, the moments of $\mathbb{E}_{t_{m-1}}[\mathcal{M}]$ are bounded by $C^{|k_m|}$. Observe that \eqref{55eq3} holds for every integer  $k_{m,l}$. Replacing $k_{m,l}$ by $k_{m,l}m$ in \textcolor{black}{\eqref{55eq3}} we get 
\begin{align} \label{55eq4}
\begin{split}
 & \Bigg| \mathbb{E} \Big[ \exp \Big( i \sum_{h=1}^m \langle \xi_h, X_{t_{h}}-X_{t_{h-1}}\rangle \Big) \Big] \Bigg|  \\
 & \quad \leq c^{\lvert mk_m\rvert}\lvert mk_m\rvert^{\textcolor{black}{(2+\delta)}\lvert mk_m\rvert} \frac{1}{(t_m-t_{m-1})^{\lvert mk_m\rvert H}}  \prod_{l=1}^d \frac{1}{|\xi_{m,l}|^{mk_{m,l}}}   .
 \end{split}
\end{align}
Now assume that $2 \leq j \leq m-1$. In this case we write

\begin{align*}
& \mathbb{E} \Big[ \exp \Big( i \sum_{h=1}^m \langle \xi_h, X_{t_{h}}-X_{t_{h-1}}\rangle \Big) \Big] \\
&  = \prod_{l=1}^d \frac{1}{(i\xi_{j,l})^{k_{j,l}}}  \\
& \times \partial_{\xi_j}^{k_j} \mathbb{E} \Bigg[   \exp \Big( i \sum_{h=1}^{j-1} \langle \xi_h, X_{t_{h}}-X_{t_{h-1}}\rangle \Big)\\
& \quad \mathbb{E}_{t_{j-1}}[\exp \big( i \langle \xi_j, X_{t_{j}}-X_{t_{j-1}}\rangle \big) \exp \Big( i \sum_{h=j+1}^m \langle \xi_h, X_{t_{h}}-X_{t_{h-1}}\rangle \Big) ] \Bigg] \\
& = \prod_{l=1}^d \frac{1}{(i\xi_{j,l})^{k_{j,l}}}  \\
& \times\mathbb{E} \Big[ \exp \Big( i \sum_{h=1}^{j-1} \langle \xi_h, X_{t_{h}}-X_{t_{h-1}}\rangle \Big) \mathbb{E}_{t_{j-1}}[\exp \big( i \langle \xi_j, X_{t_{j}}-X_{t_{j-1}}\rangle \big) H_{k_j}^{t_{j-1}} (X_{t_{j}}-X_{t_{j-1}},Y)  ]\Big] ,
\end{align*}
where we set 
$$Y =\exp \Big( i \sum_{h=j+1}^m \langle \xi_h, X_{t_{h}}-X_{t_{h-1}}\rangle \Big)$$ 
and $ H_{k_j}^{t_{j-1}} (X_{t_{j}}-X_{t_{j-1}},Y)$ as specified in Lemma \ref{ip}. Similarly as for the case $j=1$ (see equation \eqref{771}) we can prove  
\revnew{
$$\| Y\|_{\lvert k_j \rvert,q;t_{j-1}} \leq  \textcolor{black}{(1+c^{\textcolor{black}{k}}_{t_{j-1},\omega,k})} C^{\textcolor{black}{\lvert k_j \rvert^2}} \lvert k_j \rvert^{\lvert k_j \rvert} m^{\lvert k_j \rvert}  |\xi_{\gamma_j,1}|^{\lvert k_j \rvert}(t_{\gamma_j}-t_{\gamma_{j-1}})^{\lvert k_j \rvert H},$$ 
for some positive random variable $\textcolor{black}{c_{t_{j-1},\omega,k}}$ having Gaussian tails, where $\gamma_j>j$ is such that
$$  \max_{h =j+1, \ldots, m} |\xi_{h}|(t_{h}-t_{h-1})^H = |\xi_{\gamma_j}|(t_{\gamma_j}-t_{\gamma_{j-1}})^H.$$
}Then \textcolor{black}{we can again proceed as above, namely use \eqref{eq:H-norm}, estimates \eqref{eq:inverse-inc} and \eqref{eq:derivative-inc}, and the fact that $c_{t_{j-1},\omega,k}$ has Gaussian tails (together with Jensen's inequality for conditional expectations). This leads to }
\revnew{\begin{align} \label{55eq5}
\begin{split}
& \Bigg| \mathbb{E} \Big[ \exp \Big( i \sum_{h=1}^m \langle \xi_h, X_{t_{h}}-X_{t_{h-1}}\rangle \Big) \Big] \Bigg|  \\
& \quad \leq C^{\lvert k_j \rvert} \lvert k_j \rvert^{{(1+\delta)}\lvert k_j \rvert} m^{\lvert k_j \rvert} \Big(  \frac{1}{(t_j-t_{j-1})^{\lvert k_j\rvert H}}  \Big) (t_{\gamma_j}-t_{\gamma_{j-1}})^{\lvert k_j \rvert H} {|\xi_{\gamma_j,1}|^{\lvert k_j \rvert}} \prod_{l=1}^d \frac{1}{|\xi_{j,l}|^{k_{j,l}}}     .
\end{split}
\end{align}
Observe that \eqref{55eq5} holds for every integer  $k_{j,l}$. Replacing $k_{j,l}$ by $k_{j,l}m$ in \eqref{55eq5} we get 
\begin{align} \label{55eq6}
\begin{split}
& \Bigg| \mathbb{E} \Big[ \exp \Big( i \sum_{h=1}^m \langle \xi_h, X_{t_{h}}-X_{t_{h-1}}\rangle \Big) \Big] \Bigg|  \\
& \quad \leq C^{\lvert mk_j \rvert} \lvert mk_j \rvert^{(2+\delta)\lvert m k_j \rvert} \Big(  \frac{1}{(t_j-t_{j-1})^{\lvert mk_j\rvert H}}  \Big) (t_{\gamma_j}-t_{\gamma_{j-1}})^{\lvert mk_j \rvert H} {|\xi_{\gamma_j,1}|^{\lvert m k_j \rvert}} \prod_{l=1}^d \frac{1}{|\xi_{j,l}|^{m k_{j,l}}}    .
\end{split}
\end{align}}
Then combining estimates \eqref{55eq2}, \eqref{55eq4}, \eqref{55eq6} we have

\revnew{\begin{align*}
& \Bigg| \mathbb{E} \Big[ \exp \Big( i \sum_{h=1}^m \langle \xi_h, X_{t_{h}}-X_{t_{h-1}}\rangle \Big) \Big] \Bigg| ^m \\
& \leq \Big( \prod_{j=1}^{m-1} C^{m\lvert k_j \rvert} \lvert mk_j \rvert^{{(2+\delta)}\lvert m k_j \rvert} \Big) \Big( \prod_{j=1}^m  \frac{(t_{\gamma_j}-t_{\gamma_j-1})^{m\lvert k_{j}\rvert H}}{(t_j-t_{j-1})^{m\lvert k_j\rvert  H}} \Big)  \prod_{j=1}^m \frac{|\xi_{\gamma_j,1}|^{m\lvert k_{j} \rvert }}{|\xi_{j,1}|^{m k_{j,1} }} \prod_{l=2}^d \frac{1}{|\xi_{j,l}|^{m k_{j,l} }}   ,
\end{align*}
where we set $\gamma_1:=2$, which implies
\begin{align}\label{55eq7}
\begin{split}
   & \Bigg| \mathbb{E} \Big[ \exp \Big( i \sum_{h=1}^m \langle \xi_h, X_{t_{h}}-X_{t_{h-1}}\rangle \Big) \Big] \Bigg| \\
& \leq \Big( \prod_{j=1}^{m-1} C^{\lvert mk_j \rvert} \lvert mk_j \rvert^{{(2+\delta)}\lvert m k_j \rvert} \Big)^{\frac{1}{m}}\Big( \prod_{j=1}^m  \frac{(t_{\gamma_j}-t_{\gamma_j-1})^{\lvert k_{j}\rvert H}}{(t_j-t_{j-1})^{\lvert k_j\rvert  H}} \Big) \prod_{j=1}^m \frac{|\xi_{\gamma_j,1}|^{\lvert k_{j} \rvert }}{|\xi_{j,1}|^{k_{j,1} }}  \prod_{l=2}^d \frac{1}{|\xi_{j,l}|^{ k_{j,l} }}   \\
& \leq \Big( \prod_{j=1}^{m-1} C^{\lvert mk_j \rvert} \lvert mk_j \rvert^{{(2+\delta)}\lvert m k_j \rvert} \Big)^{\frac{1}{m}}\\
& \times \Bigg( \prod_{j=1}^m  \frac{1}{(t_j-t_{j-1})^{\lvert k_j\rvert  H}}  \prod_{l: \gamma_j=l}  \frac{1}{(t_l-t_{l-1})^{-\lvert k_l\rvert  H}} \Bigg) \Bigg( \prod_{j=1}^m \frac{1}{|\xi_{j,1}|^{k_{j,1} }} \prod_{l: \gamma_j=l} \frac{1}{|\xi_{l,1}|^{-\lvert k_{l}\rvert }} \Bigg)  \prod_{l=2}^d \frac{1}{|\xi_{j,l}|^{ k_{j,l} }}  
\end{split}
\end{align}
for some $C \in (0, \infty)$ not depending on $m$. Observe} that \eqref{55eq7} holds for arbitrary choice of integers $k_{j,l}$. In particular, noting that $\gamma_j>j$, this yields that
\begin{align*}
\Bigg| \mathbb{E} \Big[ \exp \Big( i \sum_{h=1}^m \langle \xi_h, X_{t_{h}}-X_{t_{h-1}}\rangle \Big) \Big] \Bigg|  \leq C^m m^{\textcolor{black}{4d{(2+\delta)}}m} \prod_{j=1}^m \Big(  \frac{1}{(t_j-t_{j-1})^{\lvert k_j\rvert H}}  \Big) \prod_{j=1}^m \prod_{l=1}^d \frac{1}{|\xi_{j,l}|^{ k_{j,l}}},
\end{align*}
where $k_{j,l} \in \{0,4\}$, which proves (A1) with {\textcolor{black}{$\theta=\frac{8d+\delta}{H}$}, for any $\delta>0$}. Moreover, condition (A2) with $\iota=\frac12$ follows from \cite[Condition $(ii)$ of Theorem 5.15, Proof of Theorem 5.16, and reference therein]{baudoin} by observing that the growth in $p$ arises from the driving Gaussian process $B^H$. Consequently, we have established (A1) for $t_1 \geq \varepsilon>0$ and (A2). Hence in order to finish the proof it suffices to prove the existence of $L(x,[0,t])$. For this, we note that we have, for any $t>\varepsilon>0$, any set $A \in \mathbb{R}^d$, and by the definition of the occupation measure $\tau_X$ and the local time,
$$
\tau_X(A,[\varepsilon,t]) = \int_A L(x,[\varepsilon,t])dx.
$$
Here $\tau_X(A,[\varepsilon,t])$ increases to $\tau_X(A,[0,t])$ and $\varepsilon \mapsto L(x,[\varepsilon,t])$ is an increasing function. Hence, for every $x$, the limit $L(x,[0,t]) := \lim_{\varepsilon\to 0} L(x,[\varepsilon,t])$ exists \footnote{\textcolor{black}{Note that the limit takes values in $\mathbb{R}\cup \{\infty\}$ a priori.}}. It follows then from the monotone convergence theorem that 
$$
\tau_X(A,[0,t]) = \int_A L(x,[0,t])dx ,
$$
so that the local time $L(x,[0,t])$ exists. This finishes the whole proof. 
\end{proof}

\begin{proof}[Proof of Proposition \ref{main3}]
We have (see Lemma 2.1, Lemma 2.2, and their proofs in \cite{knsv}) that
\begin{equation}
\label{eq:Rosenblatt-bound}
\Big\lvert  \mathbb{E} \Big[ \exp \Big( i \sum_{j=1}^m \xi_j (Z_{t_{j}}-Z_{t_{j-1}})\Big) \Big] \Big\rvert  = \prod_{k  \geq 1} \frac{1}{( 1 + 4 \lambda_k^2)^{1/4}},
\end{equation}
where
$$
\lambda_n ( B_{t, \xi}) \textcolor{black}{\geq} C(H) (\max_{1 \leq j \leq n} \lvert \xi_j\rvert  \lvert  t_j - t_{j-1}\rvert ^H ) \tilde{\mu}_n^2
$$
with $\tilde{\mu}_n \sim n^{- H/2}$. By observing that, for any integers $k_j\in\{0,4\}$, we have
\begin{align*}
&\prod_{j=1}^n \lvert \xi_j\rvert ^{k_j}\lvert t_j-t_{j-1}\rvert ^{k_j H} \\
&\leq \left(1+\max_{1\leq j\leq n}(\lvert \xi_j\rvert ^2\lvert t_j-t_{j-1}\rvert ^{2H})\right)^{\frac{n}{2}}
\end{align*}
we see that, for any constant $c_1$,
\begin{align*}
&\left(1+c_1\max_{1\leq j\leq n}(\lvert \xi_j\rvert ^2\lvert t_j-t_{j-1}\rvert ^{2H})\right)^{-\frac14}\\
&\leq c_2 \left(\prod_{j=1}^n \lvert \xi_j\rvert ^{k_j}\lvert t_j-t_{j-1}\rvert ^{k_j H}\right)^{-\frac{1}{2n}},
\end{align*}
where $c_2$ depends only on $c_1$. Hence it suffices to take $2n$-terms into account in \eqref{eq:Rosenblatt-bound} giving us condition (A1), with $\theta=0$. Similarly, (A2) with $\iota = 1$ follows from Lemma 4.1, Proposition 4.2, and their proofs in \cite{knsv}. 
\end{proof}
\begin{proof}[Proof of Proposition \ref{main4}]
For simplicity we only prove the case $d=1$, as the general case follows from this by component-wise considerations. As $Z$ is Gaussian, it follows that
\begin{align*}
\Big\lvert  \mathbb{E} \Big[ \exp \Big( i \sum_{j=1}^m \xi_j (Z_{t_{j}}-Z_{t_{j-1}})\Big) \Big] \Big\rvert  = \exp\left(-\frac12 Var\left(\sum_{j=1}^m \xi_j (Z_{t_{j}}-Z_{t_{j-1}})\right)\right).
\end{align*}
Using the local non-determinism condition \eqref{eq:Gaussian-LND}, we have
\begin{align*}
\exp\left(-\frac12 Var\left(\sum_{j=1}^m \xi_j (Z_{t_{j}}-Z_{t_{j-1}})\right)\right) \leq \prod_{j=1}^m \exp\left(-\frac{C}{2} \xi_j^2 \E(Z_{t_j}-Z_{t_{j-1}})^2\right).
\end{align*}
For (A1), it remains to apply $e^{-\lvert x\rvert } \leq \frac{C(p)}{\lvert x\rvert ^p}$ for any $p\geq 0$, with $p=0$ or $p=2$ for each term together with \eqref{eq:Gaussian-variance}. Finally, (A2) follows directly from \eqref{eq:Gaussian-variance} together with the fact that $Z$ is Gaussian. 
\end{proof}

\subsection*{Acknowledgments}
Large parts of this research were carried out while ES was visiting Aalto University School of Science and the University of Vaasa. ES would like to thank both universities for their hospitality.

\bibliographystyle{amsplain}
\bibliography{lit}

\end{document}